\documentclass[12pt,reqno]{amsart}
\usepackage{graphicx}
\usepackage{amssymb,amsmath}
\usepackage{amsthm}
\usepackage{color}
\usepackage[pdf]{pstricks}
\usepackage{hyperref}
\usepackage{empheq}
\usepackage{pgfplots}

\numberwithin{equation}{section}

\usepackage{amssymb}
\usepackage{epsfig}
\usepackage{extarrows}
\usepackage{amsfonts}
\usepackage{graphicx}
\usepackage{caption}
\usepackage{subcaption}
\usepackage{tikz}
\usepackage{color}

\newcommand{\R}{\mathbb{R}}

\newtheorem{remark}{Remark}

\newtheorem{lemma}{Lemma}
\newtheorem{theorem}{Theorem}

\begin{document}
	
\title[Instability of the peaked traveling wave]{\bf Instability of the peaked traveling wave \\ in a local model for shallow water waves}

\author{F\'{a}bio Natali}
\address[F. Natali]{Departamento de Matem\'{a}tica - Universidade Estadual de Maring\'{a}, Avenida Colombo 5790, CEP 87020-900, Maring\'{a}, PR, Brazil}
\email{fmanatali@uem.br}
	
\author{Dmitry E. Pelinovsky}
\address[D. E. Pelinovsky]{Department of Mathematics and Statistics, McMaster University, Hamilton, Ontario, Canada, L8S 4K1}
\email{pelinod@mcmaster.ca}
	
\author{Shuoyang Wang}
\address[S. Wang]{Department of Mathematics and Statistics, McMaster University, Hamilton, Ontario, Canada, L8S 4K1}
\email{wangs455@mcmaster.ca}
	
\begin{abstract}
	The traveling wave with the peaked profile arises in the limit of the family of traveling waves with the smooth profiles. We study the linear and nonlinear stability of the peaked traveling wave by using a local model for shallow water waves, which is related to the Hunter--Saxton equation. The evolution problem is well-defined in the function space $H^1_{\rm per} \cap W^{1,\infty}$, where we derive the linearized equations of motion and 
study the nonlinear evolution of co-periodic perturbations to the peaked periodic wave by using methods of characteristics. Within the linearized equations, we prove the spectral instability of the peaked traveling wave from the spectrum of the linearized operator in a Hilbert space, which completely covers the closed vertical strip with a specific half-width. Within the nonlinear equations, we prove the nonlinear instability of the peaked traveling wave by showing that the gradient of perturbations grow at the wave peak. By using numerical approximations of the smooth traveling waves and the spectrum of their associated linearized operator, we show that the spectral instability of  the peaked traveling wave cannot be obtained in the limit along the family of the spectrally stable smooth traveling waves. 
\end{abstract}

\maketitle

\section{Introduction}

Instabilities of steadily propagating waves with the periodic profiles on a fluid surface, called Stokes waves, have been recently explored within  Euler's equations in many computational details due to advanced numerical algorithms with high precision and accuracy \cite{DDLS24,DS23,DS24,Lush2}. As the Stokes waves become steeper, they become increasingly unstable with respect to co-periodic perturbations, since the spectral stability problem admits more isolated unstable eigenvalues that bifurcate from the origin due to coalescence of pairs of purely imaginary eigenvalues and splitting into pairs of real eigenvalues \cite{DS23,Lush2}. It is believed that the stability of the limiting Stokes waves with the peaked profile \cite{Amick,Plotnikov,Toland} can be concluded by studying eigenvalues of the spectral stability problem for the Stokes waves with the smooth profiles. A similar cascade of instabilities near the limit to the periodic wave with the maximal height is observed in other nonlocal wave models such as the Whitham equation \cite{Carter}. 

The purpose of this paper is to study the linear and nonlinear instability of the traveling waves with the peaked profile within a local model for evolution of surface water waves:
\begin{equation}
\label{toy-model}
2 c \eta_{tx} = (c^2 - 2 \eta) \eta_{xx} - (\eta_x)^2 + \eta,
\end{equation}
where $\eta = \eta(t,x)$ is the surface elevation and $c > 0$ is the wave speed. The subscripts denote partial derivatives of $\eta$ in $t$ and $x$. We study $2\pi$-periodic solutions in $x$ and denote the $2\pi$-periodic domain as $\mathbb{T}$ so that $\eta(t,x) : \mathbb{R}\times \mathbb{T} \to \mathbb{R}$.

The local model (\ref{toy-model}) without the last term was derived in \cite{HS91} for the dynamics of direction fields and has been referred to as the Hunter--Saxton equation \cite{HZ94}. The same model (\ref{toy-model}) was also discussed in \cite{Alber1,Alber2} in the connection to the high-frequency limit of the Camassa--Holm equation, one of the toy models for the physics of fluids with smooth and peaked traveling waves:
\begin{equation}
\label{CH}
u_{\tau} - u_{\tau \xi \xi} + k u_{\xi} + 3 u u_{\xi} = 2 u_{\xi} u_{\xi \xi} + u u_{\xi \xi \xi}, 
\end{equation}
where $u = u(\tau,\xi)$ is the horizontal velocity and $k > 0$ is the parameter. By using the transformation
\begin{equation}
\label{high-frequency}
u(\tau,\xi) = 2 \eta(t,x), \qquad t = 2c \varepsilon^{-1} \tau, \qquad 
x = \varepsilon^{-1} (\xi - c^2 \tau), \qquad k = \varepsilon^{-2},
\end{equation}
we keep only the leading-order terms at the formal $\mathcal{O}(\varepsilon^{-3})$ order. After integrating the leading-order terms once in $x$ with zero integration constant, the Camassa--Holm equation (\ref{CH}) is reduced to the Hunter--Saxton equation (\ref{toy-model}) in the high-frequency limit $\varepsilon \to 0$.

The particular form of the local model (\ref{toy-model}) was suggested in \cite{LP24} based on the reformulation of Euler's equations after a conformal transformation of the fluid domain with variable surface to a fixed rectangular domain and a formal truncation of the model near the traveling wave, see Appendix A in \cite{LP24}. In this context, $x$ is the horizontal coordinate of the rectangular domain after the conformal transformation and $t$ is the time variable defined in the traveling frame moving with the speed $c$. The traveling waves of Euler's equations correspond to the time-independent solutions of the local model (\ref{toy-model}). The local model (\ref{toy-model}) represents the nonlocal Babenko equation \cite{Babenko} for traveling waves in shallow fluid after a transformation similar 
to the high-frequency limit (\ref{high-frequency}) for the Camassa-Holm equation (\ref{CH}), see Appendix B in \cite{LP24}. 

Integrability of the local model (\ref{toy-model}) was established in \cite{Hone} together with other peaked wave equations such as the reduced Ostrovsky and short-pulse equations. Some traveling wave solutions of these peaked wave equations were studied with Hirota's bilinear method in \cite{Matsuno} and with the dynamical system methods in \cite{LP24}. Local wellposedness in Sobolev spaces for sufficiently smooth solutions has been proven in \cite{Ye}.

The Hunter--Saxton equation (\ref{toy-model}) and the Camassa--Holm equation (\ref{CH}) have the traveling periodic waves with the smooth, peaked, and cusped profiles such that the families of smooth and cusped profiles are 
connected at the limiting wave with the peaked profile \cite{GMNP,Len2,LP24}. 
In the Camassa--Holm equation (\ref{CH}), smooth traveling waves are linearly and nonlinearly stable \cite{CS-02,EJ-24,GMNP,LP-22,Len3,Long}, whereas the peaked traveling waves are linearly and nonlinearly unstable in the $W^{1,\infty}$ norm \cite{LP-21,MP-2021,Natali}, despite the fact that the perturbations do not grow in the $H^1$ norm \cite{CM,CS,Len4,Len5}. In the Hunter--Saxton equation (\ref{toy-model}), the linear stability of the smooth periodic waves was proven in \cite{LP24}. The linear and nonlinear instability of the limiting periodic wave with the peaked profile in $H^1_{\rm per} \cap W^{1,\infty}$ is the main result of the present study. Stability of the cusped traveling waves is an open problem for both models due to the lack of local well-posedness of the initial-value problem in the function spaces to which the cusped profiles belong.

For completeness, we also mention relevant results on the existence and stability of traveling periodic waves in the reduced Ostrovsky equation 
\begin{equation}
\label{rO}
(v_t + v v_x)_x = v,
\end{equation}
which is very similar to the Hunter--Saxton equation (\ref{toy-model}) rewritten in the form 
\begin{equation}
\label{HS}
(2c \eta_t - c^2 \eta_x + 2 \eta \eta_x)_x = \eta + (\eta_x)^2.
\end{equation}
Linear and nonlinear stability of the smooth traveling periodic waves were obtained for the reduced Ostrovsky equation (\ref{rO}) in \cite{GP0,HSS17,JP16}. Uniqueness of the traveling periodic waves with the peaked profiles was shown in \cite{BD19,GP1}, the results of which rule out the existence of the traveling periodic waves with the cusped profiles stated incorrectly in \cite{HSS}. The linear instability of the peaked traveling periodic waves was proven in \cite{GP1,GP2}. 

We note that the spectral stability problem $L\psi = \lambda \psi'$ with a self-adjoint operator $L$ in a Hilbert space, considered in \cite{SS18}, appears naturally for the Hunter--Saxton equation (\ref{HS}) linearized at the traveling periodic waves, with $L$ being a Hessian operator defined by the variational characterization of the traveling periodic waves, see  (\ref{action-toy}) and (\ref{linear-local}) below. In the context of the reduced Ostrovsky equation (\ref{rO}), the same spectral stability problem $L\psi = \lambda \psi'$ with a different Hessian operator $L$ is obtained after the hodograph transformation \cite{HSS17,HSS,SS18}.

In a similar context of  the cubic Novikov equation, smooth traveling waves were found to be linearly and nonlinearly stable \cite{EJL-24}, whereas the peaked traveling waves were shown to be linearly and nonlinearly unstable in the $W^{1,\infty}$ norm \cite{Chen-Pel-21,L-24} despite the perturbations do not grow in the $H^1 \cap W^{1,4}$ norm \cite{Chen-23,Chen}.

We now describe the main results and the organization of the paper. 

Section \ref{sec-2} presents the local well-posedness result in $H^1_{\rm per}(\mathbb{T}) \cap W^{1,\infty}(\mathbb{T})$ suitable for waves with the peaked profiles, see Theorem \ref{theorem-evol} below, as well as the conserved quantities useful in the analysis of stability of traveling waves with both smooth and peaked profiles. 

Section \ref{sec-3} introduces the traveling waves with the smooth and peaked profiles, see Figure \ref{fig-1}. Linearized equations of motion for the traveling waves are derived in Section \ref{sec-34}. For the smooth profiles, the spectral stability problem is equivalent to $L\psi = \lambda \psi'$ with a self-adjoint operator $L$ in a Hilbert space considered in \cite{SS18}, see equation (\ref{linear-local}). For the peaked profiles, the spectral stability problem $L\psi = \lambda \psi'$ becomes singular and the proper linearization is based on the local well-posedness result for the time evolution in  $H^1_{\rm per}(\mathbb{T}) \cap W^{1,\infty}(\mathbb{T})$, see equation (\ref{linear-evolution-final}). For the spectral theory in Hilbert spaces, it is more convenient to consider the linearized operator for the traveling wave with the peaked profile in the class of functions broader than the function space needed for the local well-posedness results. This gives us the linearized operator $A : \mathcal{D} \subset L^2(\mathbb{T}) \to L^2(\mathbb{T})$ given by (\ref{A}) and (\ref{domain-A}). 

Sections \ref{sec-4} and \ref{sec-5} contain the spectral analysis of the linearized operator $A : \mathcal{D} \subset L^2(\mathbb{T}) \to L^2(\mathbb{T})$ and its truncation to the unbounded local differential part $A_0 : \mathcal{D} \subset L^2(\mathbb{T}) \to L^2(\mathbb{T})$. In both cases, we obtain the exact location of the point spectrum and the resolvent set separated by the boundary which belongs to the spectrum, see Theorems \ref{theorem-truncated} and \ref{theorem-full-linear} below. 
Since the spectrum is located in the closed vertical strip symmetrically with respect to $i \mathbb{R}$ with a nonzero half-width of the strip, 
we conclude that the peaked traveling wave is spectrally unstable in a Hilbert space $L^2(\mathbb{T})$. 

The nonlinear instability result for the peaked traveling wave is proven in Section \ref{sec-6}, see Theorem \ref{theorem-nonlinear-instability}, by using the method of characteristics. To define the nonlinear evolution and to use the method of characteristics, we again consider perturbations to the peaked traveling wave in the function space $H^1_{\rm per}(\mathbb{T}) \cap W^{1,\infty}(\mathbb{T})$, which is a subset of the function space where the spectral instability has been proven. As a result, the nonlinear instability result is not trivially concluded from the spectral instability result. 
One of the main difficulties in establishing the nonlinear instability of the peaked traveling waves in the Hilbert space $H^1_{\rm per}(\mathbb{T})$ is that the initial-value problem associated with equation (\ref{HS}) cannot be solved using the semigroup approach for initial data in $H^1_{\rm per}(\mathbb{T})$, compared to initial data in smoother Sobolev spaces for the smooth traveling waves \cite{Ye}. To prove the nonlinear instability of the peaked traveling wave, we show that the $W^{1,\infty}$ norm of the perturbation grows in time. However, we do not know if the $H^1_{\rm per}$ norm of the perturbation grows or stays bounded, compared to the case of the Camassa--Holm equation \cite{CS,Len4}.

Section \ref{sec-7} contains numerical approximations of the periodic waves 
with the smooth profiles and eigenvalues of the corresponding Hessian operator $L$ in the spectral stability problem $L \psi = \lambda \psi'$. We show that the spectral instability of the peaked wave cannot be obtained in the limit along the family of the spectrally stable smooth waves. This further emphasizes that the stability analysis of the smooth and peaked traveling waves is very different from each other.

\section{Evolution and conserved quantities}
\label{sec-2}

Taking the mean value of the local model (\ref{HS}) for the $C^1$-smooth $2\pi$-periodic solutions $\eta \in C^1(\mathbb{R} \times \mathbb{T})$ and integrating by parts yields the constraint 
\begin{equation}
\label{zero-mean-toy}
\oint \left[ \eta + (\partial_x \eta)^2 \right] dx = 0.
\end{equation}
Let $\Pi_0 : L^2(\mathbb{T}) \rightarrow L^2(\mathbb{T}) |_{\{1\}^{\perp}}$ be a projection operator to the periodic functions with zero mean. The local model (\ref{HS}) can be written in the evolution form 
\begin{equation}
\label{local-evolution}
2 c \partial_t \eta = (c^2 - 2 \eta) \partial_x \eta + \Pi_0 \partial_x^{-1} \Pi_0 \left[ (\partial_x \eta)^2 + \eta \right],
\end{equation}
where $\Pi_0 \partial_x^{-1} \Pi_0 : L^2(\mathbb{T}) \rightarrow L^2(\mathbb{T}) |_{\{1\}^{\perp}}$ is uniquely defined on the periodic functions under the zero-mean constraint. The local well-posedness result suitable for solutions with the peaked profiles is given by the following theorem.

\begin{theorem}
	\label{theorem-evol}
	For every $\eta_0 \in H^1_{\rm per}(\mathbb{T}) \cap W^{1,\infty}(\mathbb{T})$, there exist $\tau_0 > 0$ and a unique solution $\eta \in C^0((-\tau_0,\tau_0),H^1_{\rm per}(\mathbb{T}) \cap W^{1,\infty}(\mathbb{T})) \cap C^1((-\tau_0,\tau_0),L^2(\mathbb{T}) \cap L^{\infty}(\mathbb{T}))$ of the evolution equation (\ref{local-evolution}) with $\eta(0,\cdot) = \eta_0$, 
	which is also continuous with respect to the initial data $\eta_0 \in H^1_{\rm per}(\mathbb{T}) \cap W^{1,\infty}(\mathbb{T})$.
\end{theorem}

\begin{proof}
	The evolution equation (\ref{local-evolution}) is a nonlocal version of the inviscid Burgers equation $2c \partial_t \eta = (c^2 - 2 \eta) \partial_x \eta$. Since 
\begin{align*}
\| \Pi_0 \partial_x^{-1} \Pi_0 \left[ \eta + (\partial_x \eta)^2 \right] \|_{L^2} &\leq \| \eta  + (\partial_x \eta)^2 \|_{L^2} \leq \| \eta \|_{L^2} + \| \partial_x \eta \|_{L^{\infty}} \| \partial_x \eta \|_{L^2}, \\
\| \Pi_0 \partial_x^{-1} \Pi_0 \left[ \eta + (\partial_x \eta)^2 \right] \|_{L^{\infty}} &\leq \| \eta + (\partial_x \eta)^2 \|_{L^1} \leq \sqrt{2\pi} \| \eta \|_{L^2} + \| \partial_x \eta \|^2_{L^2},\\
\| \partial_x \Pi_0 \partial_x^{-1} \Pi_0 \left[ \eta + (\partial_x \eta)^2 \right]  \|_{L^2 \cap L^{\infty}} &\leq \|  \eta + (\partial_x \eta)^2 \|_{L^2 \cap L^{\infty}} \leq \| \eta \|_{L^2 \cap L^{\infty}} + \| \partial_x \eta \|_{L^{\infty}} \| \partial_x \eta \|_{L^2 \cap L^{\infty}},
\end{align*}
the nonlocal term $\Pi_0 \partial_x^{-1} \Pi_0 \left[ \eta + (\partial_x \eta)^2 \right]$ is a bounded operator from a ball in $H^1_{\rm per}(\mathbb{T}) \cap W^{1,\infty}(\mathbb{T})$ to $H^1_{\rm per}(\mathbb{T}) \cap W^{1,\infty}(\mathbb{T})$. Local well-posedness in $H^1_{\rm per}(\mathbb{T}) \cap W^{1,\infty}(\mathbb{T})$ follows by using the method of characteristics. 
\end{proof}

\begin{remark}
	The same argument can be used to establish the local well-posedness 
	of the evolution equation (\ref{local-evolution}) in smooth Sobolev spaces  $H^s_{\rm per}(\mathbb{T})$ with $s > \frac{3}{2}$, see \cite{Ye}. 
	The smooth Sobolev spaces are continuously embedded into the function space $H^1_{\rm per}(\mathbb{T}) \cap W^{1,\infty}(\mathbb{T})$.
\end{remark}

The mass, momentum, and energy conservation of the evolution equation (\ref{local-evolution}) are obtained for the smooth solution $\eta \in C^0((-\tau_0,\tau_0),H^s_{\rm per}(\mathbb{T}))$ with $s > \frac{3}{2}$. Multiplying (\ref{toy-model}) by $\partial_x \eta$, and integrating over the period, yields conservation of the momentum 
\begin{equation}
\label{Q}
Q(\eta) := \frac{1}{2} \oint (\partial_x \eta)^2 dx,
\end{equation}
and in view of the constraint (\ref{zero-mean-toy}), also conservation of the mass 
\begin{equation}
\label{M}
M(\eta) := \oint \eta dx.
\end{equation}
Furthermore, writing (\ref{local-evolution}) in the Hamiltonian form
\begin{equation}
\label{action-toy}
2c \partial_t \eta = - \Pi_0 \partial_x^{-1} \Pi_0 \left[ c^2 Q'(\eta) - H'(\eta) \right], 
\end{equation}
where
\begin{equation}
\label{H}
H(\eta) := \frac{1}{2} \oint \left[ \eta^2 + 2 \eta (\partial_x \eta)^2 \right] dx,
\end{equation}
yields conservation of the energy $H(\eta)$.

\begin{remark}
	Due to integrability of the local model (\ref{toy-model}), higher-order conserved quantities exist. Nevertheless, conservation of $Q(\eta)$, $M(\eta)$, and $H(\eta)$ is sufficient for the existence and stability analysis of the traveling waves with the smooth and peaked profiles. 
\end{remark}

\section{Traveling wave with the smooth and peaked profiles}
\label{sec-3}

A traveling wave with the speed $c$ and the smooth profile $\eta \in C_{\rm per}^{\infty}(\mathbb{T})$ corresponds to the time-independent solution of the local model (\ref{toy-model}) found from the second-order differential equation 
\begin{equation}
\label{TW-eq}
(c^2-2 \eta) \eta'' - (\eta')^2 + \eta = 0, \quad x \in \mathbb{T}. 
\end{equation}
This equation is integrable with the first-order invariant
\begin{equation}
\label{TW-inv}
E(\eta,\eta') = \frac{1}{2} (c^2-2 \eta) (\eta')^2 + \frac{1}{2} \eta^2 = \mathcal{E}, 
\end{equation}
the value of which is independent of $x$.

It was shown in \cite{LP24} that the family of smooth $2\pi$-periodic solutions  $\eta \in C_{\rm per}^{\infty}(\mathbb{T})$  exists for $c \in (1,c_*)$ with $c_* := \frac{\pi}{2 \sqrt{2}}$. The peaked profile  $\eta_* \in C^0_{\rm per}(\mathbb{T}) \cap W^{1,\infty}(\mathbb{T})$ corresponds to $c = c_*$ and is given explicitly as 
\begin{equation}
\label{quadratic}
\eta_*(x) = \frac{1}{16} ( \pi^2 - 4 \pi |x| + 2 x^2), \qquad x \in [-\pi,\pi],
\end{equation}
extended as a $2\pi$-periodic function on $\mathbb{T}$. It is easy to verify the validity of $\eta_*$ in (\ref{quadratic}) as a solution of (\ref{TW-eq}) for $x \in [-\pi,0) \cup (0,\pi]$ with 
$$
\max_{x \in \mathbb{T}} \eta_*(x) = \eta_*(0) = \frac{c_*^2}{2}.
$$ 
The peaked profile $\eta_* \in C^0_{\rm per}(\mathbb{T}) \cap W^{1,\infty}(\mathbb{T})$ corresponds to the marginal value of $\mathcal{E}_c := \frac{c^4}{8}$ in (\ref{TW-inv}), which separates the smooth profiles for $\mathcal{E}\in (0,\mathcal{E}_c)$ and the cusped profiles 
for $\mathcal{E} \in (\mathcal{E}_c,\infty)$. The value of $c = c_*$ is selected by setting the period of the peaked profile to $2\pi$. 
The slope of the peaked profile $\eta_*$ has a finite jump discontinuity at $x = 0$ since 
\begin{equation}
\label{wave-slope}
\eta_*'(x) = -\frac{1}{4} (\pi - |x|) {\rm sign}(x) \qquad x \in [-\pi,\pi],
\end{equation}
which implies that $\eta_*'(0^+) - \eta_*'(0^-) = -\frac{\pi}{2}$. 
By using the Dirac delta distribution $\delta_0$ at $x = 0$, we can express the finite jump discontinuity of $\eta'_*(x)$ as the Dirac delta singularity of the second derivative at $x = 0$:
\begin{equation}
\label{wave-curvature}
\eta_*''(x) = \frac{1}{4} - \frac{\pi}{2} \delta_0, \qquad x \in [-\pi,\pi].
\end{equation}
It can be checked through explicit computations that the periodic solution with the peaked profile (\ref{quadratic}) satisfies the constraint (\ref{zero-mean-toy}).

Figure \ref{fig-1} presents the periodic profiles $\eta$ of the traveling waves for two values of $c$ in $(1,c_*)$ and for $c = c_*$ (left) as well as the dependence of the wave amplitude $\| \eta \|_{L^{\infty}}$ versus $c$ (right). The wave profiles were approximated numerically, see Section \ref{sec-7}. The peaked profile $\eta_*$ is shown by a dashed line on the left panel and the corresponding value $c_*$ is shown by a dashed vertical line on the right panel. The family of periodic waves is continued past $c = c_*$ with the cusped profiles for $c \in (c_*,c_{\infty})$, where $c_{\infty}$ is numerically computed (dashed-dotted line) \cite{LP24}. 

\begin{figure}[htb!]
	\centering
	\includegraphics[width=0.95\textwidth]{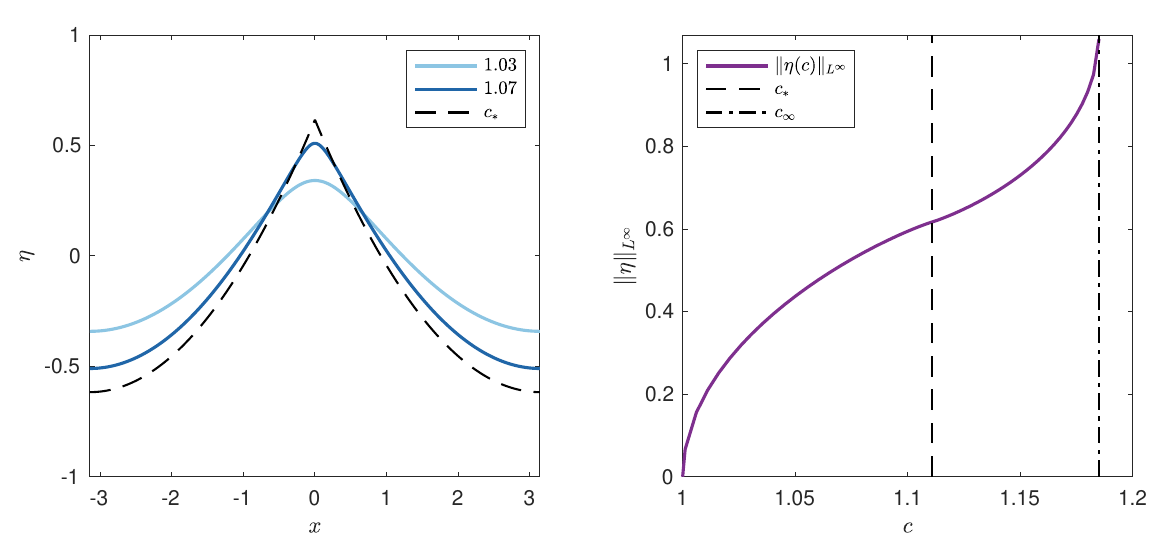} 
	\caption{(a) The solid lines represent the smooth profiles $\eta$ for $c = 1.03, 1.07$. The dashed line represents the peaked profile $\eta_*$ for $c = c_*$. (b) The wave amplitude versus the wave speed $c$ for smooth profiles in $(1,c_*)$ and cusped profiles in $(c_*,c_{\infty})$, where $c_*\approx 1.1107$ (dashed line) and $c_\infty\approx 1.1850$ (dashed-dotted line).}
		\label{fig-1}
\end{figure}

\begin{remark}
	In the context of Babenko's equation \cite{Babenko} for the fluid of infinite depth, it is shown in \cite{LP25} that the peaked profiles with the local behavior as in (\ref{wave-slope}) do not exist after the conformal transformation of the fluid domain. Nevertheless, the existence of the peaked profiles is well established for the local evolution equations such as the Hunter--Saxton equation (\ref{toy-model}), the Camassa--Holm equation (\ref{CH}), and the reduced Ostrovsky equation (\ref{rO}) with $x$ being the horizontal coordinate of the original fluid domain.
\end{remark}

\section{Linearization at the smooth and peaked  traveling waves}
\label{sec-34}

Let $\eta \in C_{\rm per}^{\infty}(\mathbb{T})$ be the spatial profile of the smooth traveling waves for $c \in (1,c_*)$. The second-order equation (\ref{TW-eq}) is equivalent to the Euler--Lagrange equation 
$$
H'(\eta) - c^2 Q'(\eta) = 0. 
$$
Adding a perturbation $\zeta(t,x)$ to $\eta(x)$ and linearizing the evolution equation (\ref{action-toy}), we obtain the linearized equation in the form
\begin{equation}\label{linear-local}
2c\partial_t \zeta = - \Pi_0 \partial_x^{-1} \Pi_0 \mathcal L \zeta, \quad \quad \mathcal L := -\partial_x (c^2-2\eta)\partial_x  + (2\eta'' -1), 
\end{equation}
where $\mathcal L : H_\mathrm{per}^2(\mathbb T) \subset L^2 (\mathbb T)\to L^2(\mathbb T)$ is the Hessian operator for $c^2 Q(\eta) - H(\eta)$ at the profile $\eta \in C_{\rm per}^{\infty}(\mathbb{T})$. 
As $c \to c_*$ and $\eta \to \eta_*  \in C^0_{\rm per}(\mathbb{T}) \cap W^{1,\infty}(\mathbb{T})$, the Hessian operator becomes singular 
since 
$$
2 \eta_*''(x) - 1 = -\frac{1}{2} - \pi \delta_0, \quad x \in [-\pi,\pi].
$$
This suggests that the linearized equation (\ref{linear-local}) breaks at the peaked traveling wave. We need to be careful to linearize the evolution equation (\ref{action-toy}) about the traveling wave with the peaked profile $\eta_*$ by working in the function space $H^1_{\rm per}(\mathbb{T}) \cap W^{1,\infty}(\mathbb{T})$, where the local well-posedness is established by Theorem \ref{theorem-evol}.

To get the proper linearization near the peaked profile  $\eta_*$, we note the following result, which is obtained verbatim from the analysis of \cite{MP-2021,Natali}.

\begin{lemma}\cite{MP-2021,Natali}
	\label{lem-char} 
	Consider a local solution $\eta \in C^0((-\tau_0,\tau_0),H^1_{\rm per}(\mathbb{T}) \cap W^{1,\infty}(\mathbb{T}))$ of Theorem \ref{theorem-evol}, and assume that there exists $\xi(t)$ such that 
	$$
	\lim_{x \to \xi(t)^-} \partial_x \eta(t,x) \neq \lim_{x \to \xi(t)^+} \partial_x \eta(t,x), \quad t \in (-\tau_0,\tau_0).
	$$
	Then, $\xi \in C^1((-\tau_0,\tau_0))$ satisfies 
	\begin{equation}
	\label{char-speed}
	\xi'(t) = -\frac{1}{2c} (c^2 - 2 \eta(t,\xi(t))), \quad t \in (-\tau_0,\tau_0).
	\end{equation}
\end{lemma}

In order to consider a local solution $\eta \in C^0((-\tau_0,\tau_0),H^1_{\rm per}(\mathbb{T}) \cap W^{1,\infty}(\mathbb{T}))$ of the evolution equation 
(\ref{local-evolution}) in a local neighborhood of the traveling wave with the peaked profile (\ref{quadratic}), we define the decomposition 
\begin{equation}
\label{decomp}
\eta(t,x) = \eta_*(x-\xi(t)) + \zeta(t,x-\xi(t)),
\end{equation}
where we assume that the only peak of $\eta(t,\cdot)$ on $\mathbb{T}$ is located at $x = \xi(t)$ for some $t \in (-\tau_0,\tau_0)$. 
The peak location $\xi(t)$ moves with the local characteristic speed of the inviscid Burgers equation, as in Lemma \ref{lem-char}. Substituting 
(\ref{decomp}) into (\ref{local-evolution}) with $c = c_*$ and using (\ref{char-speed}) yields the evolution problem for the perturbation term $\zeta(t,x)$:
\begin{equation}
\label{evolution-perturbation}
2 c_* \partial_t \zeta = (c_*^2 - 2 \eta_*) \partial_x \zeta 
-2 (\zeta - \zeta |_{x = 0}) (\eta_*' + \partial_x \zeta) + \Pi_0 \partial_x^{-1} \Pi_0 \left[ \zeta + 2 \eta_*' \partial_x \zeta + (\partial_x \zeta)^2 \right],
\end{equation}
where we have translated $x - \xi(t)$ into $x$ on $\mathbb{T}$.
Truncation of the evolution equation (\ref{evolution-perturbation}) by the linear terms in $\zeta$ yields the linearized equation 
\begin{equation}
\label{linear-evolution}
2 c_* \partial_t \zeta = (c_*^2 - 2 \eta_*) \partial_x \zeta 
-2 \eta_*' (\zeta - \zeta |_{x = 0}) + \Pi_0 \partial_x^{-1} \Pi_0 \left[ \zeta + 2 \eta_*' \partial_x \zeta  \right],
\end{equation}
subject to the linearized constraint 
\begin{equation}
\label{linear-constraint}
\oint [\zeta + 2 \eta_*' \partial_x \zeta ] dx = 0.
\end{equation}
The next result gives the equivalent form of the linearized evolution. 

\begin{lemma}
	\label{lem-linear}
	The linearized equation (\ref{linear-evolution}) is equivalently written in the form 
\begin{equation}
\label{linear-evolution-final}
2 c_* \partial_t \zeta = (c_*^2 - 2 \eta_*) \partial_x \zeta 
- \frac{1}{\pi} \oint \eta_*' \zeta dx 
+ \frac{1}{2} \Pi_0 \partial_x^{-1} \Pi_0 \zeta,
\end{equation}	
	where both $\oint \zeta dx$ and $\zeta |_{x = 0}$ are constant in $t$ and satisfies the constraint
\begin{equation}
\label{zeta-0}
\zeta |_{x = 0} = -\frac{1}{2\pi} \oint \zeta dx,
\end{equation}	
\end{lemma}

\begin{proof}
The constraint (\ref{zeta-0}) is obtained by subtituting  (\ref{wave-slope}) into (\ref{linear-constraint}), and integrating the second term in (\ref{linear-constraint}) by parts. To simplify the linearized equation (\ref{linear-evolution}), we use (\ref{wave-curvature}) and write
\begin{equation}
\label{help-1}
\zeta + 2 \eta_*' \partial_x \zeta = 2 \partial_x (\eta_*' \zeta) + \frac{1}{2} \zeta + \pi \delta_0 \zeta.
\end{equation}
This transforms the linearized equation (\ref{linear-evolution}) to the form 
\begin{equation}
\label{linear-evolution-new}
2 c_* \partial_t \zeta = (c_*^2 - 2 \eta_*) \partial_x \zeta 
- \frac{1}{\pi} \oint \eta_*' \zeta dx 
+ 2 \eta_*' \zeta |_{x = 0} + \frac{1}{2} \Pi_0 \partial_x^{-1} \Pi_0 \zeta 
+ \pi \zeta |_{x = 0} \Pi_0 \partial_x^{-1} \Pi_0 \delta_0.
\end{equation}
By using Fourier series, we get 
$$
\delta_0 = \frac{1}{2\pi} \sum_{n \in \mathbb{Z}} e^{i n x},
$$
so that 
$$
\Pi_0 \partial_x^{-1} \Pi_0 \delta_0 = \Pi_0 \partial_x^{-1} \left( \delta_0 -\frac{1}{2\pi} \right) = \sum_{n \in \mathbb{Z}\backslash \{0\}} \frac{1}{2 \pi i n} e^{i n x}.
$$
On the other hand, it follows from (\ref{wave-curvature}) that 
$$
\eta_*''(x) = \frac{1}{4} - \frac{\pi}{2} \delta_0 = -\frac{1}{4} \sum_{n \in \mathbb{Z}\backslash \{0\}} e^{i n x},
$$
which yields
$$
\eta_*'(x) = -\sum_{n \in \mathbb{Z}\backslash \{0\}} \frac{1}{4in} e^{i n x}.
$$
Hence, the two terms with $\zeta |_{x=0}$ in (\ref{linear-evolution-new}) cancels out as
\begin{equation}
\label{help-2}
2 \eta_*' + \pi \Pi_0 \partial_x^{-1} \Pi_0 \delta_0 = 0, 
\end{equation}
and the linearized equation (\ref{linear-evolution-new}) can be written in the form (\ref{linear-evolution-final}). 

Finally, we show that both $\oint \zeta dx$ and $\zeta |_{x = 0}$ are constant in $t$. The conservation of $\oint \zeta dx$ follows by taking the mean value of (\ref{linear-evolution-final}), with the account of the projection term $-\frac{1}{\pi} \oint \eta_*' \zeta dx$. The conservation of $\zeta |_{x = 0}$ is shown by taking the limit $x \to 0$ since if $\zeta = \sum_{n \in \mathbb{Z}} \zeta_n e^{inx}$, then 
\begin{align}
\notag
\frac{1}{\pi} \oint \eta_*' \zeta dx &= \frac{1}{\pi} \sum_{n \in \mathbb{Z}\backslash \{0\}} \zeta_n \oint \eta_*'(x)e^{inx} dx \\
\notag
&= -\frac{1}{\pi}\sum_{n \in \mathbb{Z}\backslash \{0\}} \zeta_n \left( \sum_{m \in \mathbb{Z}\backslash \{0\}} \frac{1}{4im} \oint e^{i(n+m)x} dx \right) \\
&= \sum_{n \in \mathbb{Z} \backslash \{0\}} \frac{\zeta_n}{2in}
\label{help-3}
\end{align}
and 
\begin{align}
\label{help-4}
\lim_{x \to 0} \frac{1}{2} \Pi_0 \partial_x^{-1} \Pi_0 \zeta = \lim_{x \to 0} \frac{1}{2} \sum_{n \in \mathbb{Z}\backslash \{0\}} \zeta_n  \Pi_0 \partial_x^{-1} e^{inx} = \sum_{n \in \mathbb{Z}\backslash \{0\}} \frac{\zeta_n}{2in},
\end{align}
from which it follow that $2c \lim\limits_{x \to 0} \partial_t \zeta(t,x) = 0$ and the value of $\zeta |_{x=0}$ is preserved in time.
\end{proof}

The linearized equation (\ref{linear-evolution-final}) of Lemma \ref{lem-linear} is defined by the linearized operator $A : {\rm Dom}(A) \subset L^2(\mathbb{T}) \to L^2(\mathbb{T})$ given by 
\begin{equation}
\label{A}
A f := (c_*^2 - 2 \eta_*) \partial_x f - \frac{1}{\pi} \oint \eta_*' f dx + \frac{1}{2} \Pi_0 \partial_x^{-1} \Pi_0 f,
\end{equation}
where 
\begin{equation}
\label{domain-A}
{\rm Dom}(A) := \left\{ f \in L^2(\mathbb{T}) : \;\; (c_*^2 - 2 \eta_*) f' \in L^2(\mathbb{T}) \right\} \equiv \mathcal{D}.
\end{equation} 

\begin{remark}
	\label{rem-2}
The local well-posedness result of Theorem \ref{theorem-evol} 
suggests that we should consider the linear operator 
$A : H^1_{\rm per}(\mathbb{T}) \cap W^{1,\infty}(\mathbb{T}) \subset 
L^2(\mathbb{T}) \cap L^{\infty}(\mathbb{T}) \to L^2(\mathbb{T}) \cap L^{\infty}(\mathbb{T})$ with the same definition of $A$ as in (\ref{A}). 
However, for the spectral stability theory, it is more convenient to work in a Hilbert space $L^2(\mathbb{T})$ for which the domain of $A$ is given by 
(\ref{domain-A}). 
\end{remark}

We denote the spectrum of $A : \mathcal{D} \subset L^2(\mathbb{T}) \to L^2(\mathbb{T})$ by $\sigma(A)$. According to the standard definition 
(Definition 6.1.9 in \cite{B}), the spectrum $\sigma(A)$ is further divided into three disjoint sets of the point spectrum $\sigma_p(A)$, the residual spectrum $\sigma_r(A)$, and the continuous spectrum $\sigma_c(A)$ with the resolvent set denoted by $\rho(A) = \mathbb{C} \backslash \sigma(A)$. 

\begin{remark}
By using (\ref{zero-mean-toy}), (\ref{wave-slope}), (\ref{wave-curvature}), and $c_* = \frac{\pi}{2 \sqrt{2}}$, we obtain  
\begin{align*}
A \eta_*' &= (c_*^2 - 2 \eta_*) \eta_*'' - \frac{1}{\pi} \oint (\eta_*')^2 dx + \frac{1}{2} \Pi_0 \eta_* \\
&= -\frac{\pi}{2} (c_*^2 - 2 \eta_*) \delta_0 + \frac{1}{4} (c_*^2 - 2 \eta_*) - \frac{1}{\pi} \oint (\eta_*')^2 dx + \frac{1}{2} \eta_* 
- \frac{1}{4\pi} \oint \eta_* dx \\
&= -\frac{\pi}{2} (c_*^2 - 2 \eta_*) \delta_0 + \frac{\pi^2}{32}  - \frac{3}{4 \pi} \oint (\eta_*')^2 dx \\
&= -\frac{\pi}{2} (c_*^2 - 2 \eta_*) \delta_0,
\end{align*}
so that $A \eta_*' = 0$	in $L^2(\mathbb{T})$. Similarly, we have $\eta_*' \in \mathcal{D}$ so that $0 \in \sigma_p(A)$. However, $\eta_*' \notin C^0_{\rm per}(\mathbb{T})$, hence 
	$\eta_*' \notin H^1_{\rm per}(\mathbb{T}) \cap W^{1,\infty}(\mathbb{T})$. Thus, $H^1_{\rm per}(\mathbb{T}) \cap W^{1,\infty}(\mathbb{T})$ is embedded into $\mathcal{D}$ but is not equivalent to $\mathcal{D}$. The spectral theory of the linear operator  $A : \mathcal{D} \subset L^2(\mathbb{T}) \to L^2(\mathbb{T})$ is developed in a wider space of functions than the space needed for the local well-posedness of the evolution equation  (\ref{local-evolution}).
\end{remark}

\section{Truncated linearized equation}
\label{sec-4}

The following lemma allows us to truncate the linearized operator (\ref{A})--(\ref{domain-A}).

\begin{lemma}
	\label{lem-compact}
The linear operator $K := \frac{1}{2} \Pi_0 \partial_x^{-1} \Pi_0 : L^2(\mathbb{T}) \to L^2(\mathbb{T})$ is a compact (Hilbert--Schmidt) operator.
\end{lemma}

\begin{proof}
By the Fourier theory, we have 
$$
\sigma(K) = \sigma_p(K) = \left\{ \frac{1}{2n}, \;\; n \in \mathbb{Z} \backslash \{0\} \right\}.
$$
Since eigenvalues of $\sigma_p(K)$ are square summable, $K : L^2(\mathbb{T}) \to L^2(\mathbb{T})$ is a compact (Hilbert--Schmidt) operator.
\end{proof}

Using Lemma \ref{lem-compact}, we see that $A = A_0 + K$, where $K$ is a compact perturbation of the unbounded truncated operator $A_0 : {\rm Dom}(A_0) \subset L^2(\mathbb{T}) \to L^2(\mathbb{T})$ given by 
\begin{equation}
\label{A-0}
A_0 f := (c_*^2 - 2 \eta_*) \partial_x f - \frac{1}{\pi} \oint \eta_*'(x) f(x) dx,
\end{equation}
with the same ${\rm Dom}(A_0) = {\rm Dom}(A) = \mathcal{D}$. 
The constraint in the definition of $A_0$ ensures that 
\begin{equation}
	\label{constraint-A}
\oint (A_0f)(x) dx = 0 \quad \mbox{\rm if} \;\; f \in \mathcal{D}.
\end{equation}
The spectrum of $A_0$ can be analyzed similar to the work \cite{GP2}. In fact, since 
$$
c_*^2 - 2 \eta_*(x) = \frac{1}{4} [\pi^2 - (\pi - |x|)^2], \quad x \in [-\pi,\pi],
$$
extended as a $2\pi$-periodic function on $\mathbb{T}$, we define the change of coordinates $x \mapsto z$ by 
\begin{equation}
\label{coord}
\frac{dx}{dz} = \frac{1}{4} x (2\pi - x), \quad x \in [0,2\pi],
\end{equation}
where the interval $[0,2\pi]$ is located between the two consequent peaks on the periodic domain $\mathbb{T}$. Solving the differential equation (\ref{coord}) with $x(0) = \pi$ yields 
\begin{equation}
\label{coord-z}
x(z) = \pi + \pi \tanh\left( \frac{\pi z}{4} \right),
\end{equation}
which is an invertible mapping $\R \ni z \mapsto x \in [0,2\pi]$.
The following lemma shows that the spectrum of the operator 
$A_0 : \mathcal{D}\subset L^2(\mathbb{T}) \to L^2(\mathbb{T})$ can be found from the spectrum of a simpler linear operator defined on the infinite line $\mathbb{R}$.

\begin{lemma}
	\label{lem-reformulation}
	The spectrum of the truncated operator $A_0 : \mathcal{D}\subset L^2(\mathbb{T}) \to L^2(\mathbb{T})$ is equivalent to the spectrum of 
	the linear operator $D_0 : H^1(\R) \subset L^2(\R) \to L^2(\R)$ given by 
	\begin{equation}
	\label{D-0}
	D_0 h := \partial_z h + \frac{\pi}{4} \tanh\left(\frac{\pi z}{4} \right) h + \frac{\pi}{4} w(z) \int_{\R} w'(z) h(z) dz,
	\end{equation}
	where $w(z) := {\rm sech}\left(\frac{\pi z}{4} \right)$. 
	The constraint (\ref{constraint-A}) is equivalent to the constraint $\langle w, D_0 h \rangle = 0$, which holds for every $h \in H^1(\R)$, where $\langle \cdot, \cdot \rangle$ is the standard inner product in $L^2(\R)$ with the induced norm $\| \cdot \|$.
\end{lemma}

\begin{proof}
Using the transformation (\ref{coord-z}), we obtain $A_0 f = B_0 g$, 
where $g(z) = f(x)$ and $B_0 : {\rm Dom}(B_0) \subset L_w^2(\R) \to L_w^2(\R)$ is given by 
\begin{equation}
\label{B-0}
B_0 g := \partial_z g + \frac{\pi}{4} \int_{\R} w(z) w'(z) g(z) dz
\end{equation}
and
$$
{\rm Dom}(B_0) := \left\{ g \in L_w^2(\R) : \quad g' \in L_w^2(\R)  \right\} \equiv H^1_w(\R),
$$ 
with the weight $w(z) := {\rm sech}\left(\frac{\pi z}{4} \right)$. Here the  exponentially weighted spaces $L^2_w(\R)$ and $H^1_w(\R)$ are defined with the squared norms:
$$
\| g \|_{L^2_w}^2 = \int_{\R} w^2(z) |g(z)|^2 dz, \qquad 
\| g \|_{H^1_w}^2 = \int_{\R} w^2(z) \left( |g'(z)|^2 + |g(z)|^2 \right) dz
$$
and the inner product in $L^2_w(\R)$ is defined by 
$$
\langle g_1,g_2 \rangle_{L^2_w} = \int_{\R} w^2(z) g_1(z) g_2(z) dz.
$$
The constraint (\ref{constraint-A}) is equivalent to the constraint $\langle 1, B_0 g \rangle_{L^2_w} = 0$, which holds for every $g \in H^1_w(\R)$.
By using the change of variables $h(z) = w(z) g(z)$, we get 
$B_0 g = w^{-1} D_0 h$ and $\langle 1, B_0 g \rangle_{L^2_w} = \langle w, D_0 h \rangle = 0$, where $D_0 : H^1(\R) \subset L^2(\R) \to L^2(\R)$ is given by (\ref{D-0}). 
\end{proof}

The following theorem prescribes the spectrum of the truncated operator $A_0 : \mathcal{D}\subset L^2(\mathbb{T}) \to L^2(\mathbb{T})$ given by (\ref{A-0}).

\begin{theorem}
	\label{theorem-truncated}
	The spectrum of $A_0 : \mathcal{D} \subset L^2(\mathbb{T}) \to L^2(\mathbb{T})$ completely covers the closed vertical strip given
	by
\begin{equation}
\label{spectrum-truncated}
\sigma(A_0) = \left\{ \lambda \in \mathbb{C} : \quad -\frac{\pi}{4} \leq {\rm Re}(\lambda) \leq \frac{\pi}{4} \right\}.
\end{equation}
\end{theorem}

\begin{proof}
	We obtain $\sigma_p(A_0)$ and $\rho(A_0)$ as
\begin{align}
\sigma_p(A_0) &= \left\{ \lambda \in \mathbb{C} : \quad -\frac{\pi}{4} < {\rm Re}(\lambda) < \frac{\pi}{4} \right\}, 
\label{sigma-p-A0} \\
\rho(A_0) &= \left\{ \lambda \in \mathbb{C} : \quad |{\rm Re}(\lambda)| >  \frac{\pi}{4} \right\}.
\label{rho-A0}
\end{align}		
Since $\sigma(A_0)$ is a closed set and $\rho(A_0)$ is an open set, the closure of the open region (\ref{sigma-p-A0}) yields (\ref{spectrum-truncated}). 

\vspace{0.25cm}

\underline{$\sigma_p(A_0)$:} By Lemma \ref{lem-reformulation}, it is equivalent to consider $\sigma_p(D_0)$, where $D_0 : H^1(\R) \subset L^2(\R) \to L^2(\R)$ is given by (\ref{D-0}). Let $h \in H^1(\R)$ be a solution of $D_0 h = \lambda h$ for some $\lambda \in \mathbb{C}$. Then, $h = h(z)$ satisfies 
\begin{equation}
\label{eq-1}
h'(z) + \frac{\pi}{4} \tanh\left(\frac{\pi z}{4} \right) h(z) + \frac{\pi}{4} w(z) \langle w', h \rangle = \lambda h(z), \quad z \in \mathbb{R}, 
\end{equation}	
subject to the orthogonality condition $\lambda \langle w, h \rangle = 0$.
Substitution $h(z) = \tilde{h}(z) w(z)$ reduces (\ref{eq-1}) to the form 
\begin{equation}
\label{eq-2}
\tilde{h}'(z) + \frac{\pi}{4} \langle w w', \tilde{h} \rangle = \lambda \tilde{h}(z), \quad z \in \mathbb{R},
\end{equation}	
subject to the orthogonality condition $\lambda \langle w^2, \tilde{h} \rangle = 0$.

For $\lambda = 0$, the general solution of equation (\ref{eq-2}) is 
$\tilde{h}(z) = c_1 + c_2 z$, where $(c_1,c_2)$ are arbitrary constants. This yields the general solution $h(z) = (c_1 + c_2 z) w(z)$ of equation (\ref{eq-1}) for $\lambda = 0$. Since $h \in H^1(\R)$, then $0 \in \sigma_p(D_0)$. 

For $\lambda \neq 0$, the general solution of equation (\ref{eq-2}) is a scalar multiplier of the particular solution 
\begin{equation}
\label{tilde-h}
\tilde{h}(z) = e^{\lambda z} + \frac{\pi}{4 \lambda} \langle w w', e^{\lambda z} \rangle,
\end{equation}
where the inner product is defined for $|{\rm Re}(\lambda)| < \frac{\pi}{2}$. Since 
$$
\frac{\pi}{8} \| w \|^2 = \frac{1}{2} \int_{\mathbb{R}} {\rm sech}^2(z) dz = 1, 
$$
the orthogonality condition $\lambda \langle w^2, \tilde{h} \rangle = 0$ is satisfied for the solution (\ref{tilde-h}). Integration by parts and transformation back to $h$ yields the solution 
$$
h(z) = e^{\lambda z} w(z) - \frac{\pi}{8} w(z) \langle w^2, e^{\lambda z} \rangle,
$$
which show that $h \in H^1(\R)$ if and only if $|{\rm Re}(\lambda)| < \frac{\pi}{4}$, so that $\sigma_p(A_0) = \sigma_p(D_0)$ is given by (\ref{sigma-p-A0}).

\vspace{0.25cm}

\underline{$\rho(A_0)$:} By Lemma \ref{lem-reformulation}, it is equivalent to consider $\rho(D_0)$. Let $h \in H^1(\R)$ be a solution of $(D_0 - \lambda) h = f$ for some $\lambda \in \mathbb{C}$ and $f \in L^2(\R)$. Then, $h = h(z)$ satisfies 
\begin{equation}
\label{eq-3}
h'(z) + \frac{\pi}{4} \tanh\left(\frac{\pi z}{4} \right) h(z) + \frac{\pi}{4} w(z) \langle w', h \rangle = \lambda h(z) + f(z), \quad z \in \mathbb{R}.
\end{equation}
Since $\langle w, D_0 h \rangle = 0$, we have $\lambda \langle w, h \rangle + \langle w, f \rangle = 0$. 	Multiplying equation (\ref{eq-3}) by $\bar{h}$, integrating over $\R$, adding complex conjugation, and dividing by $2$ yields 
$$
\frac{\pi}{4} \langle \tanh\left(\frac{\pi z}{4} \right) h,h \rangle 
+ \frac{\pi}{4} {\rm Re} \langle h,w \rangle \langle w',h \rangle = {\rm Re}(\lambda) \| h \|^2 + {\rm Re} \langle h,f \rangle.
$$
By Cauchy--Schwarz inequality and the constraint $\bar{\lambda} \langle h,w  \rangle + \langle f,w \rangle = 0$, we obtain 
\begin{align*}
\left( {\rm Re}(\lambda) - \frac{\pi}{4} \right) \| h \|^2 &\leq 
{\rm Re}(\lambda) \| h \|^2 - \frac{\pi}{4} \langle \tanh\left(\frac{\pi z}{4} \right) h,h \rangle \\
&= - {\rm Re} \langle h,f \rangle - {\rm Re} \frac{\pi}{4 \bar{\lambda}} \langle f,w \rangle \langle w',h \rangle \\
&\leq \left(1 + \frac{\pi \| w \| \|w'\|}{4 |\lambda|} \right) \| h \| \| f \|
\end{align*}
and 
\begin{align*}
\left( -{\rm Re}(\lambda) - \frac{\pi}{4} \right) \| h \|^2 &\leq
-{\rm Re}(\lambda) \| h \|^2 + \frac{\pi}{4} \langle \tanh\left(\frac{\pi z}{4} \right) h,h \rangle \\
&= {\rm Re} \langle h,f \rangle + {\rm Re} \frac{\pi}{4 \bar{\lambda}} \langle f,w \rangle \langle w',h \rangle  \\
&\leq \left(1 + \frac{\pi \| w \| \|w'\|}{4 |\lambda|} \right) \| h \| \| f \|.
\end{align*}
This yields the bound 
$$
\| h \| \leq \left( 1 + \| w \| \|w'\| \right) \frac{\| f \|}{|{\rm Re}(\lambda)| - \frac{\pi}{4}}, \quad 
\text{for} \;\; |{\rm Re}(\lambda)| > \frac{\pi}{4}.
$$	
Hence, $\{ \lambda \in \mathbb{C}: |{\rm Re}(\lambda)| > \frac{\pi}{4} \}$ belongs to $\rho(D_0)$, but since $\sigma(D_0)$ is a closed set and $\rho(D_0)$ is an open set, then $\{ \lambda \in \mathbb{C}: |{\rm Re}(\lambda)| > \frac{\pi}{4} \}$ is equivalent to $\rho(D_0)$ in view of the location of $\sigma_p(D_0)$. This completes the proof of $\rho(A_0) = \rho(D_0)$ given by (\ref{rho-A0}). 
\end{proof}

\section{Full linearized evolution}
\label{sec-5}

The full linearized evolution is defined by the linear operator 
$A : \mathcal{D} \subset L^2(\mathbb{T}) \to L^2(\mathbb{T})$ given by 
(\ref{A}). The following theorem prescribes the spectrum of $A$.

\begin{theorem}
	\label{theorem-full-linear}
	The spectrum of $A : \mathcal{D} \subset L^2(\mathbb{T}) \to L^2(\mathbb{T})$ completely covers the closed vertical strip given
	by
	\begin{equation}
	\label{spectrum-A}
	\sigma(A) = \left\{ \lambda \in \mathbb{C} : \quad -\frac{\pi}{4} \leq {\rm Re}(\lambda) \leq \frac{\pi}{4} \right\}.
	\end{equation}
\end{theorem}

\begin{proof}
Again, we obtain $\sigma_p(A)$ and $\rho(A)$ as
\begin{align}
\sigma_p(A) &= \left\{ \lambda \in \mathbb{C} : \quad -\frac{\pi}{4} < {\rm Re}(\lambda) < \frac{\pi}{4} \right\} = \sigma_p(A_0), 
\label{sigma-p-A} \\
\rho(A) &= \left\{ \lambda \in \mathbb{C} : \quad |{\rm Re}(\lambda)| >  \frac{\pi}{4} \right\} = \rho(A_0).
\label{rho-A}
\end{align}		
Since $\sigma(A)$ is a closed set and $\rho(A)$ is an open set, the closure of the open region (\ref{sigma-p-A}) yields (\ref{spectrum-A}).
	
	\vspace{0.25cm}
	
\underline{$\sigma_p(A)$:}	Let $f \in \mathcal{D}$ be a solution of $Af = \lambda f$ for some $\lambda \in \mathbb{C}$. Then, $f = f(x)$ satisfies 
\begin{equation}
\label{de-f}
\frac{1}{4} x(2\pi - x) f'(x) + \frac{1}{4 \pi} \int_{0}^{2\pi} (\pi-x) f(x) dx + \frac{1}{2} \Pi_0 \partial_x^{-1} \Pi_0 f = \lambda f(x), \quad 0 < x < 2\pi.
\end{equation}
subject to the orthogonality condition $\lambda \int_0^{2\pi} f(x) dx = 0$. 

If $f \in \mathcal{D}$, then $f \in L^2(\mathbb{T})$ and $(c_*^2-2\eta_*) f' \in L^2(\mathbb{T})$ so that $\lim\limits_{x \to 0^+} f(x)$ and $\lim\limits_{x \to 2\pi^-} f(x)$ may not be defined. Nevertheless, since $c_*^2 - 2 \eta_*(x) \neq 0$ for $x \in (0,2\pi)$, 
we have $f \in C^0(0,2\pi)$. Bootstrapping arguments for equation (\ref{de-f}) imply that $f \in C^{\infty}(0,2\pi)$. Differentiating (\ref{de-f}) in $x$ yields the second-order differential equation 
\begin{equation}
\label{de-f-second-order}
\frac{1}{4} x(2\pi - x) f''(x) + \frac{1}{2} (\pi - x) f'(x) + \frac{1}{2} f(x) 
- \frac{1}{4\pi} \int_0^{2\pi} f(x) dx = \lambda f'(x), \quad 0 < x < 2\pi.
\end{equation}
If $\lambda \neq 0$, then $\int_0^{2\pi} f(x) dx = 0$ so that equation (\ref{de-f-second-order}) can be rewritten in the form
\begin{equation}
\label{de-f-second-order-new}
\frac{1}{4} x(2\pi - x) f''(x) + \frac{1}{2} (\pi - x) f'(x) + \frac{1}{2} f(x) 
= \lambda f'(x), \quad 0 < x < 2\pi.
\end{equation}
One solution of (\ref{de-f-second-order-new}) is obtained explicitly as $f_1(x) = 2 \lambda - \pi + x$. 
The second linearly independent solution $f_2(x)$ of (\ref{de-f-second-order-new}) is obtained from the Wronskian 
\begin{equation}
\label{Wronskian}
f_1(x) f_2'(x) - f_1'(x) f_2(x) = W(x), 
\end{equation}
where $W(x)$ satisfies the first-order differential equation by Abel's theorem:
\begin{equation} 
\label{Wronskian-eq}
W'(x) = \frac{2 (2 \lambda - \pi + x)}{x(2\pi-x)} W(x).
\end{equation}
Integrating (\ref{Wronskian-eq}) yields 
\begin{equation}
\label{W}
W(x) = \frac{\pi^2}{x (2\pi-x)} \left( \frac{x}{2 \pi - x} \right)^{\frac{2\lambda}{\pi}},
\end{equation}
where the constant of integration has been normalized by the condition $W(\pi) = 1$. It follows from (\ref{Wronskian}) and (\ref{W}) that if $\lambda \neq 0$, then $f_2(x)$ satisfies the following asymptotic limits
\begin{align*}
f_2(x) \sim \left\{ \begin{array}{ll} x^{\frac{2\lambda}{\pi}}, \quad & \lambda \neq \frac{\pi}{2}, \\
1, \quad & \lambda = \frac{\pi}{2} \end{array}
\right. \quad \text{ as } \;\; x \to 0^+
\end{align*}
and 
\begin{align*}
f_2(x) \sim \left\{ \begin{array}{ll} (2\pi - x)^{-\frac{2\lambda}{\pi}}, \quad & \lambda \neq -\frac{\pi}{2}, \\
1, \quad & \lambda = -\frac{\pi}{2} \end{array}
\right. \quad \text{ as } \;\; x \to (2\pi)^-.
\end{align*}
On the other hand, $f_1,f_2 \in C^{\infty}(0,2\pi)$ for every $\lambda \in \mathbb{C}$. 

If $|{\rm Re}(\lambda)| < \frac{\pi}{4}$, then $f_2 \in L^2(\mathbb{T})$ and $(c_*^2-\eta_*) f_2' \in L^2(\mathbb{T})$, so that 
$f(x) = c_1 f_1(x) + c_2 f_2(x)$ belongs to $\mathcal{D}$ for every $(c_1,c_2) \in \R^2$. Satisfying the constraint  $\int_0^{2\pi} f(x) dx = 0$ for $\lambda \neq 0$ yields a one-parameter family of solutions $f \in \mathcal{D}$ of (\ref{de-f}) for every $\lambda \in \sigma_p(A) \backslash \{0\}$, where $\sigma_p(A)$ is given by (\ref{sigma-p-A}). 

If $\lambda = 0$, then equation (\ref{de-f-second-order}) contains a constant term. Without the constant term $-\frac{1}{4\pi} \int_0^{2\pi} f(x) dx$, 
the two homogeneous solutions $f_1,f_2 \in \mathcal{D}$ are given by 
$$
f_1(x) = x - \pi, \quad f_2(x) = \frac{1}{2\pi} (x-\pi) \ln \frac{x}{2\pi - x} - 1, 
$$
However, only $f_1$ satisfies (\ref{de-f-second-order}) since $\int_0^{2\pi} f_1(x) dx = 0$. On the other hand, it is easy to see that $f(x) = 1$ is also a solution of (\ref{de-f-second-order}). Thus, there exists a two-parameter family of solutions $f = c_1 (x-\pi) + c_2 \in \mathcal{D}$ of equation (\ref{de-f}) for $\lambda = 0$, so that $0 \in \sigma_p(A)$.

Finally, if $|{\rm Re}(\lambda)| \geq \frac{\pi}{4}$, then $f_2 \notin L^2(\mathbb{T})$ due to the asymptotic limits as $x \to 0^+$ and $x \to (2\pi)^-$. On the other hand, $\int_0^{2\pi} f_1(x) dx = 4\pi \lambda \neq 0$ for $\lambda \neq 0$, so that there exist no nonzero solutions $f \in \mathcal{D}$ of equation (\ref{de-f}) for every $\lambda \notin \sigma_p(A)$. 
This completes the proof of $\sigma_p(A)$ given by (\ref{sigma-p-A}).

\vspace{0.25cm}

\underline{$\rho(A)$:} Let $f \in \mathcal{D}$ be a solution of the resolvent equation $(A - \lambda) f = g$ for $g \in L^2(\mathbb{T})$ rewritten in the form:
\begin{equation}
\label{resolvent}
(c_*^2 - 2 \eta_*) \partial_x f -\frac{1}{\pi} \oint \eta_*' f dx
+ \frac{1}{2} \Pi_0 \partial_x^{-1} \Pi_0 f - \lambda f = g, \quad x \in \mathbb{T}. 
\end{equation}
Since $\oint Af dx = 0$, we get $-\lambda \oint f dx = \oint g dx$. Taking into account that $\Pi_0 \partial_x^{-1} \Pi_0$ is a skew-adjoint operator in $L^2(\mathbb{T})$, we multiply (\ref{resolvent}) by $\bar{f}$, integrate over $\mathbb{T}$, add a complex conjugate equation, and divide by $2$ to obtain
$$
{\rm Re} \langle f, (c_*^2 - 2 \eta_*) \partial_x f \rangle 
- {\rm Re}(\lambda) \| f \|^2 
-\frac{1}{\pi} {\rm Re} \langle f, 1 \rangle \langle \eta_*', f \rangle  = {\rm Re} \langle f,g \rangle.
$$
%where $\langle \cdot,\cdot \rangle$ is the standard inner product in $L^2(\mathbb{T})$ with the induced norm $\| \cdot \|$. 
Integrating by parts in the first term 
and using $-\bar{\lambda} \langle f, 1 \rangle = \langle g, 1 \rangle$ for $\lambda \neq 0$, we obtain 
$$
\langle \eta_*' f,f \rangle 
- {\rm Re}(\lambda) \| f \|^2 + {\rm Re} \frac{\langle \eta_*', f \rangle \langle g, 1 \rangle}{\pi \bar{\lambda}} = {\rm Re} \langle f,g \rangle.
$$
Since $-\frac{\pi}{4} \leq \eta_*'(x) \leq \frac{\pi}{4}$ for $x \in [0,2\pi]$, we get by Cauchy--Schwarz inequality that 
\begin{align*}
\left( {\rm Re}(\lambda) - \frac{\pi}{4} \right) \| f \|^2 &\leq 
{\rm Re}(\lambda) \| f \|^2 - \langle \eta_*' f,f \rangle \\
&= -{\rm Re} \langle f,g \rangle + {\rm Re} \frac{\langle \eta_*', f \rangle \langle g, 1 \rangle}{\pi \bar{\lambda}} \\
&\leq \left(1 + \frac{\sqrt{2\pi} \| \eta_*' \|}{\pi |\lambda|} \right) \| g \| \| f \|
\end{align*}
and 
\begin{align*}
\left( -{\rm Re}(\lambda) - \frac{\pi}{4} \right) \| f \|^2 &\leq 
-{\rm Re}(\lambda) \| f \|^2 + \langle \eta_*' f,f \rangle \\
&={\rm Re} \langle f,g \rangle - {\rm Re} \frac{\langle \eta_*', f \rangle \langle g, 1 \rangle}{\pi \bar{\lambda}}  \\
&\leq \left(1 + \frac{\sqrt{2\pi} \| \eta_*' \|}{\pi |\lambda|} \right) \| g \| \| f \|.
\end{align*}
This yields the bound 
$$
\| f \| \leq \left(1 + \frac{4 \sqrt{2\pi} \| \eta_*' \|}{\pi^2} \right) \frac{\| g \|}{|{\rm Re}(\lambda)| - \frac{\pi}{4}}, \quad 
\text{for} \;\; |{\rm Re}(\lambda)| > \frac{\pi}{4}.
$$
Hence, $\{ \lambda \in \mathbb{C}: |{\rm Re}(\lambda)| > \frac{\pi}{4} \}$ belongs to $\rho(A)$, but since $\sigma(A)$ is a closed set and $\rho(A)$ is an open set, then $\{ \lambda \in \mathbb{C}: |{\rm Re}(\lambda)| > \frac{\pi}{4} \}$ is equivalent to $\rho(A)$ in view of the location of $\sigma_p(A)$ in (\ref{sigma-p-A}). This completes the proof of $\rho(A)$ given by (\ref{rho-A}). 
\end{proof}

\begin{remark}
	Since the intersections of $\sigma_p(A) \cap \rho(A_0)$ and $\sigma_p(A_0) \cap \rho(A)$ are empty, as follows from (\ref{sigma-p-A0}), (\ref{rho-A0}), (\ref{sigma-p-A}), and (\ref{rho-A}), the result $\sigma(A) = \sigma(A_0)$ also follows by Theorem 1 in \cite{GP2}.
Computations in the proof of Theorem \ref{theorem-full-linear} do not rely on the truncated operator $A_0 : \mathcal{D} \subset L^2(\mathbb{T}) \to L^2(\mathbb{T})$ introduced and studied in Section \ref{sec-4}. We included these computations anyway to emphasize that the linear instability of the peaked traveling wave is induced by the quasilinear part of the inviscid Burgers equation and that the dispersion term of the Hunter--Saxton equation (\ref{HS}) does not play the role. This has been the main property of the linear instability of the peaked traveling wave in the Camassa--Holm equation (\ref{CH}) \cite{LP-21,MP-2021,Natali}, the reduced Ostrovsky equation (\ref{rO}) \cite{GP1,GP2}, and the Novikov equation \cite{Chen-Pel-21,L-24}.
\end{remark}

\section{Nonlinear evolution}
\label{sec-6}

We shall now derive the nonlinear instability result for the peaked traveling wave by using the nonlinear evolution equation (\ref{evolution-perturbation}). With the help of equations (\ref{help-1}) and (\ref{help-2}) in Lemma \ref{lem-linear}, we rewrite the nonlinear evolution equation (\ref{evolution-perturbation}) in the equivalent form:
\begin{equation}
\label{evolution-nonlinear}
2 c_* \partial_t \zeta = (c_*^2 - 2 \eta_*) \partial_x \zeta 
-2 (\zeta - \zeta |_{x = 0}) \partial_x \zeta 
-\frac{1}{\pi} \oint \eta_*' \zeta dx 
+ \frac{1}{2} \Pi_0 \partial_x^{-1} \Pi_0 \left[ \zeta + 2 (\partial_x \zeta)^2 \right].
\end{equation}
After applying the transformation (\ref{decomp}) and integrating by parts using (\ref{wave-curvature}), the original constraint (\ref{zero-mean-toy}) becomes
\begin{align}
0 = \oint \left[ \zeta + 2 \eta_*' \partial_x \zeta + (\partial_x \zeta)^2 \right] dx = \frac{1}{2} \oint \left[ \zeta + 2 (\partial_x \zeta)^2 \right] dx + \pi \zeta|_{x=0}.
\label{constraint-nonlinear}
\end{align}
The following lemma identifies two conserved quantities of the evolution equation (\ref{evolution-nonlinear}), whose sum yields the constraint  (\ref{constraint-nonlinear}). 

\begin{lemma}
	\label{lem-conserved}
	Consider a local solution $\zeta \in C^0((-\tau_0,\tau_0),H^1_{\rm per}(\mathbb{T}) \cap W^{1,\infty}(\mathbb{T}))$ of the evolution equation (\ref{evolution-nonlinear}) for some $\tau_0 > 0$. Then 
	\begin{equation}
	\label{conserved}
	\oint \zeta dx \quad \mbox{\rm and} \quad \zeta|_{x=0} + \frac{1}{\pi} \oint (\partial_x \zeta)^2 dx 
	\end{equation}
	are conserved for $t \in (-\tau_0,\tau_0)$.
\end{lemma}

\begin{proof}
	The conservation of $\oint \zeta dx$ is shown by taking the mean value of the evolution equation (\ref{evolution-nonlinear}), while the conservation of $\zeta|_{x=0} + \frac{1}{\pi} \oint (\partial_x \zeta)^2 dx$ follows from the constraint (\ref{constraint-nonlinear}). We can also show the latter conservation directly as follows. For smooth solutions, we differentiate equation (\ref{evolution-nonlinear}) and obtain 
\begin{align}
\notag
2 c_* \partial_t \partial_x \zeta &= (c_*^2 - 2 \eta_*) \partial^2_x \zeta 
-2 \eta_*' \partial_x \zeta - 2 (\zeta - \zeta |_{x = 0}) \partial^2_x \zeta 
- 2(\partial_x \zeta)^2 \\
& \quad + \frac{1}{2} (\zeta + 2 (\partial_x \zeta)^2)  - \frac{1}{4\pi} \oint (\zeta + 2 (\partial_x \zeta)^2) dx
\label{evolution-derivative}
\end{align}	
Multiplying (\ref{evolution-derivative}) by $\partial_x \zeta$ and integrating in $x$ over $\mathbb{T}$ yields
\begin{equation}
\label{equality-1}
c_* \frac{d}{dt} \oint (\partial_x \zeta)^2 dx = - \oint \eta_*' (\partial_x \zeta)^2 dx.
\end{equation}
Furthermore, taking the limit $x \to 0$ in equation (\ref{evolution-nonlinear}) as in (\ref{help-3}) and (\ref{help-4}), we get 
\begin{equation}
\label{equality-2}
2 c_* \lim_{x\to 0} \partial_t \zeta = \lim_{x \to 0} \Pi_0 \partial_x^{-1} \Pi_0 (\partial_x \zeta)^2 = \frac{2}{\pi} \oint \eta_*' (\partial_x \zeta)^2 dx.
\end{equation}
By adding (\ref{equality-1}) divided by $\pi$ and (\ref{equality-2}) divided by $2$, we verify conservation of $\zeta|_{x=0} + \frac{1}{\pi} \oint (\partial_x \zeta)^2 dx$. By Sobolev's embedding of $H^1_{\rm per}(\mathbb{T})$ into $C^0_{\rm per}(\mathbb{T})$, we have well-defined $\zeta|_{x=0} \in C^0(-\tau_0,\tau_0)$. Hence, the conserved quantities (\ref{conserved}) are well-defined for the local solution 
$\zeta \in C^0((-\tau_0,\tau_0),H^1_{\rm per}(\mathbb{T}) \cap W^{1,\infty}(\mathbb{T}))$.
\end{proof}

Based on the conserved quantity $\zeta|_{x=0} + \frac{1}{\pi} \oint (\partial_x \zeta)^2 dx$, we obtain the nonlinear instability of 
the peaked traveling wave given by the following theorem.

\begin{theorem}
	\label{theorem-nonlinear-instability}
	For every $\delta > 0$ there exists $\zeta_0 \in H^1_{\rm per}(\mathbb{T}) \cap W^{1,\infty}(\mathbb{T})$ satisfying
\begin{equation}
\label{bound-initial}
	\| \zeta_0 \|_{H^1_{\rm per}} \leq \delta^2, \quad 	\| \zeta_0 \|_{W^{1,\infty}} \leq \delta,
\end{equation}
such that the unique local solution $\zeta \in C^0((-\tau_0,\tau_0),H^1_{\rm per}(\mathbb{T}) \cap W^{1,\infty}(\mathbb{T}))$ of the evolution equation (\ref{evolution-nonlinear}) with $\zeta|_{t = 0} = \zeta_0$ satisfies 
\begin{equation}
\label{bound-final}
\| \zeta(t_0,\cdot) \|_{W^{1,\infty}} = 1,
\end{equation}
for some $t_0 \in (0,\tau_0)$.
\end{theorem}

\begin{proof}
	The proof follows by the method of characteristics which works for every local solution $\zeta \in C^0((-\tau_0,\tau_0),H^1_{\rm per}(\mathbb{T}) \cap W^{1,\infty}(\mathbb{T}))$ of the evolution equation (\ref{evolution-nonlinear}). Let $x = X(t,s)$ be the family of characteristic curves for $(t,s) \in (-\tau_0,\tau_0) \times (0,2\pi)$ obtained from 
\begin{equation}
\label{ivp-1}
\left\{ \begin{array}{l} 2 c_* \partial_t X(t,s) = -(c_*^2 - 2 \eta_*(X)) + 2 (\zeta(t,X) - \zeta(t,0) ), \\
X(0,s) = s. \end{array} \right.
\end{equation}
Since the vector field of the initial-value problem (\ref{ivp-1}) is Lipschitz for the local solution $\zeta \in C^0((-\tau_0,\tau_0),H^1_{\rm per}(\mathbb{T}) \cap W^{1,\infty}(\mathbb{T}))$, 
there is a unique solution for $X \in C^1((-\tau_0,\tau_0) \times (0,2\pi))$ 
such that $X(t,0) = 0$ and $X(t,2\pi) = 2 \pi$. Since $[0,2\pi]$ is the fundamental period of $\mathbb{T}$, we also get $\zeta(t,0) = \zeta(t,2\pi)$.
By solving the linear equation for $\partial_s X(t,s)$, we get 
$$
\partial_s X(t,s) = \exp\left( \frac{1}{c_*} \int_0^t \left[ \eta_*'(X(t',s)) + \partial_x \zeta(t',X(t',s)) \right] dt' \right),
$$
from which it follows that $\partial_s X(t,s) > 0$ for every $t \in (-\tau_0,\tau_0)$ and $s \in (0,2\pi)$. Hence the mapping  $[0,2\pi] \ni s \mapsto X(t,s) \in [0,2\pi]$ is a diffeomorphism for $t \in (-\tau_0,\tau_0)$.

To proceed further, we consider the restriction of $\zeta_0 \in H^1_{\rm per}(\mathbb{T}) \cap W^{1,\infty}(\mathbb{T}) \cap  C^1(0,2\pi)$. Along the family of characteristic curves, we can now define $Z(t,s) := \zeta(t,X(t,s))$ and $V(t,s) := \partial_x \zeta(t,X(t,s))$. By using the evolution equations (\ref{evolution-nonlinear}) and (\ref{evolution-derivative}) along the family of characteristic curves satisfying (\ref{ivp-1}), we obtain the following initial-value problems:
\begin{equation}
\label{ivp-2}
\left\{ \begin{array}{l} 2 c_* \partial_t Z(t,s) = -\frac{1}{\pi} \langle \eta_*', \zeta \rangle + \frac{1}{2} \Pi_0 \partial_x^{-1} \Pi_0 (\zeta + 2 (\partial_x \zeta)^2), \\
Z(0,s) = \zeta_0(s), \end{array} \right.
\end{equation}
and
\begin{equation}
\label{ivp-3}
\left\{ \begin{array}{l} 2 c_* \partial_t V(t,s) = -2\eta_*'(X) V - V^2 + \frac{1}{2} (Z(t,s) + Z(t,0)), \\
V(0,s) = \zeta_0'(s), \end{array} \right.
\end{equation}
where we have used the constraint (\ref{constraint-nonlinear}). 
If $\zeta_0 \in H^1_{\rm per}(\mathbb{T}) \cap W^{1,\infty}(\mathbb{T}) \cap  C^1(0,2\pi)$, then the solutions of the initial-value problems (\ref{ivp-1}), (\ref{ivp-2}) and (\ref{ivp-3}) are defined in the class of functions 
$$
\begin{cases} 
X \in C^1((-\tau_0,\tau_0) \times (0,2\pi)), \\ 
Z \in C^1((-\tau_0,\tau_0) \times (0,2\pi)), \\
V \in C^1((-\tau_0,\tau_0),C^0(0,2\pi)), 
\end{cases}
$$
respectively, with the bounded one-sided limits as $s \to 0^+$ and $s \to (2\pi)^-$. Due to the conservation law in (\ref{conserved}), we have 
$$
\zeta|_{x=0} + \frac{1}{\pi} \oint (\partial_x \zeta)^2 dx = C_0,
$$ 
from which $\zeta |_{x=0} \leq C_0$ with a time-independent positive 
constant $C_0$. Since $X(t,0) = 0$, we get $Z(t,0) \leq C$ so that the evolution problem (\ref{ivp-3}) in the limit $s \to 0^+$ yields 
for $V_0(t) := \lim\limits_{s \to 0^+} V(t,s)$:
\begin{align*}
2c_* V_0'(t) = \frac{\pi}{2} V_0(t) - V_0^2(t) + Z(t,0) \leq \frac{\pi}{2} V_0(t) + C_0.
\end{align*}
Iterating the inequality as 
\begin{align*}
2c_* \frac{d}{dt} e^{-\frac{\pi t}{4c_*}} V_0(t) \leq C_0  e^{-\frac{\pi t}{4c_*}}
\end{align*}
and integrating yields 
\begin{align*}
2c_* \left[ e^{-\frac{\pi t}{4c_*}} V_0(t) - V_0(0) \right] \leq \frac{4 c_* C_0}{\pi} \left[ 1 - e^{-\frac{\pi t}{4 c_*}} \right] \leq \frac{4 c_* C_0}{\pi}.
\end{align*}
This implies  
$$
V_0(t) \leq \left( V_0(0) + \frac{2}{\pi} C_0 \right) e^{\frac{\pi t}{4 c_*}}.
$$
If the initial bound (\ref{bound-initial}) is true, we get by Sobolev embeddings of $H^1_{\rm per}(\mathbb{T})$ to $L^{\infty}(\mathbb{T})$ that 
$$
|C_0| \leq \delta^2 + \frac{1}{\pi} \delta^2 \leq 2 \delta^2,
$$ 
so that the interval $(-\delta,-\frac{2}{\pi} |C_0|)$ is nonempty for all sufficiently small $\delta > 0$. Selecting $-\delta < V_0(0) < -\frac{2}{\pi} |C_0|$ to satisfy the initial bound (\ref{bound-initial}) 
and to ensure that $V_0(0) + \frac{2}{\pi} C_0 < 0$ yields the exponential divergence $V_0(t) \to -\infty$ as $t \to +\infty$. Since $\tau_0 > 0$ is the maximal existence time in $H^1_{\rm per}(\mathbb{T}) \cap W^{1,\infty}(\mathbb{T})$ norm, there exists $t_0 \in (0,\tau_0)$ 
such that the instability bound (\ref{bound-final}) holds. 
\end{proof}

\begin{remark}
	By incorporating the quadratic term in the bound 
	$$
	2c_* V_0'(t) = \frac{\pi}{2} V_0(t) - V_0^2(t) + Z(t,0) \leq \frac{\pi}{2} V_0(t) - V_0^2(t) + C_0,
	$$
	one can find initial data  $\zeta_0 \in H^1_{\rm per}(\mathbb{T}) \cap W^{1,\infty}(\mathbb{T})$ for which $\| \zeta(t,\cdot) \|_{W^{1,\infty}}$ diverges in a finite time, see \cite{MP-2021} for a similar analysis of the Camassa--Holm equation. However, we do not have the bound on $\| \zeta(t,\cdot) \|_{H^1_{\rm per}}$ compared to the case of the Camassa--Holm equation, where the $H^1_{\rm per}$ norm of the perturbation does not grow, see \cite{Len4,Len5,MP-2021}.
\end{remark}

\section{Numerical approximations}
\label{sec-7}

We approximate numerically the smooth profile $\eta \in C_{\rm per}^{\infty}(\mathbb{T})$ of the traveling waves from the second-order equation (\ref{TW-eq}) and the eigenvalues of the Hessian operator 
$\mathcal L : H_\mathrm{per}^2(\mathbb T) \subset L^2 (\mathbb T)\to L^2(\mathbb T)$ which defines the linearized time evolution (\ref{linear-local}) for the smooth traveling waves. In both cases, 
we are interested to understand the convergence of numerical results to the peaked traveling wave
with the profile $\eta_*  \in C^0_{\rm per}(\mathbb{T}) \cap W^{1,\infty}(\mathbb{T})$ as $c \to c_*$. We show 
that the lowest eigenvalue of $\mathcal{L}$ diverges as $c \to c_*$, 
which suggests that the linearized equation (\ref{linear-local}) cannot be used for the peaked traveling wave. This explains why we had to derive a different linearized equation (\ref{linear-evolution}) for the peaked traveling wave.

To obtain the solutions $\eta \in C_{\rm per}^{\infty}(\mathbb{T})$ of the second-order equation (\ref{TW-eq}) for $c \in (1,c_*)$, we use the first-order invariant (\ref{TW-inv}) and define $\eta$ from the boundary-value problem:
\begin{equation}\label{bvp}
\left\{
\begin{array}{ll}
\displaystyle \left (\frac{d\eta}{dx} \right )^2=\frac{2\mathcal E-\eta^2}{c^2-2\eta}, \\[7pt]
\eta(\pm\pi)= -\sqrt{2\mathcal E}.
\end{array}
\right.
\end{equation}
Since $\eta(-x) = \eta(x)$, we take the negative sign in the square root of (\ref{bvp}) for $x \in [0,\pi]$ and obtain the solution profile $\eta(x)$ for $x \in [0,\pi]$ by finding the root of the integral equation
\begin{equation}
\label{int-eq}
f(\eta(x))-x = 0 , \qquad f(\eta) :=
\int_{\eta/\sqrt{2\mathcal E}}^{1}\frac{\sqrt{c^2-2\sqrt{2\mathcal E}x}}{\sqrt{1-x^2}}\,dx, 
\end{equation}
for $\mathcal E \in (0,\mathcal E_c)$,
where $\mathcal E_c := \frac{c^4}{8}$ is the value separating smooth and cusped profiles.

To determine the value of $\mathcal E$ for each $c\in (1,c_*)$, we consider the period function $T(\mathcal E,c)$ studied in \cite{LP24}. The period function is represented by using the complete elliptic integral $E(\kappa)$ of the second kind with the elliptic modulus $\kappa \in (0,1)$, which is defined by $\mathcal{E}\in (0,\mathcal{E}_c)$ and $c \in (1,c_*)$ as follows: 
\begin{equation}\label{period}
T(\mathcal E,c)=
4E(\kappa)\sqrt{c^2+2\sqrt{2\mathcal E}}, \quad 
\kappa = \sqrt{\frac{4\sqrt{2\mathcal E}}{c^2+2\sqrt{2\mathcal E}}}.
\end{equation}
The periodic profile $\eta \in C_{\rm per}^{\infty}(\mathbb{T})$ corresponds to the value of $\mathcal{E} = \mathcal{E}(c)$ found from the root of $T(\mathcal{E}(c),c) = 2\pi$.

To approximate the solution profile numerically, we let $x_j = jh,j\in \{0,\dots ,N\}$ be a fixed grid with the step size $h=\pi/N$ for a large integer $N$. By the fundamental theorem of calculus, the solution profile $\{(x_j,\eta_j)\}_{j=0}^{N}$ can be found by solving $f(\eta_j)-x_j=0$ for every $j$. We implement the Newton–Raphson's method as the root-finding algorithm under a specific tolerance $\epsilon>0$, such that
\begin{equation}\label{newton}
\eta_j^{(k+1)} = \eta_j^{(k)} - \frac{1}{f'\left (\eta_j^{(k)}\right )} \left [ f\left (\eta_j^{(k)}\right )-u_{j}\right ], \quad \quad \left  | f\left (\eta_j^{(k)}\right )-u_{j} \right |< \epsilon,
\end{equation}
where $k\in \mathbb N$ denotes iteration number, and we take $\eta_0^{(k)}(0)=\sqrt{2\mathcal E}$ and $\eta_N^{(k)}(\pi)= -\sqrt{2\mathcal E}$ as two boundary grid points for any $k$ to avoid the singularities. Given a $c$-grid $\{c_i\}_{i=1}^M$ with $M$ grid points, 
we compute  $\mathcal{E}_i:=\mathcal E(c_i)$ from the period function (\ref{period}) by solving numerically $T(\mathcal{E}_i,c_i) = 2\pi$. This outputs the sets of parameter pairs $\{(c_i, \mathcal E_i)\}_{i=1}^M$, which can be used in the root-finding algorithm (\ref{newton}). After the solution $\{(x_j,\eta_j)\}_{j=0}^{N}$ is obtained on $[0,\pi]$ for $N+1$ grid points, the even reflection fills all $2N+1$ grid points on $[-\pi,\pi]$ domain. The solution points $\{(x_j,\eta_j)\}_{j=-N}^{N}$ are plotted in Figure \ref{fig-1} (left) and the parameter pairs $\{(c_i, \mathcal E_i)\}_{i=1}^M$ are plotted in Figure \ref{fig-1} (right). 

\begin{figure}[htp!]
	\centering
	\includegraphics[width=0.75\textwidth]{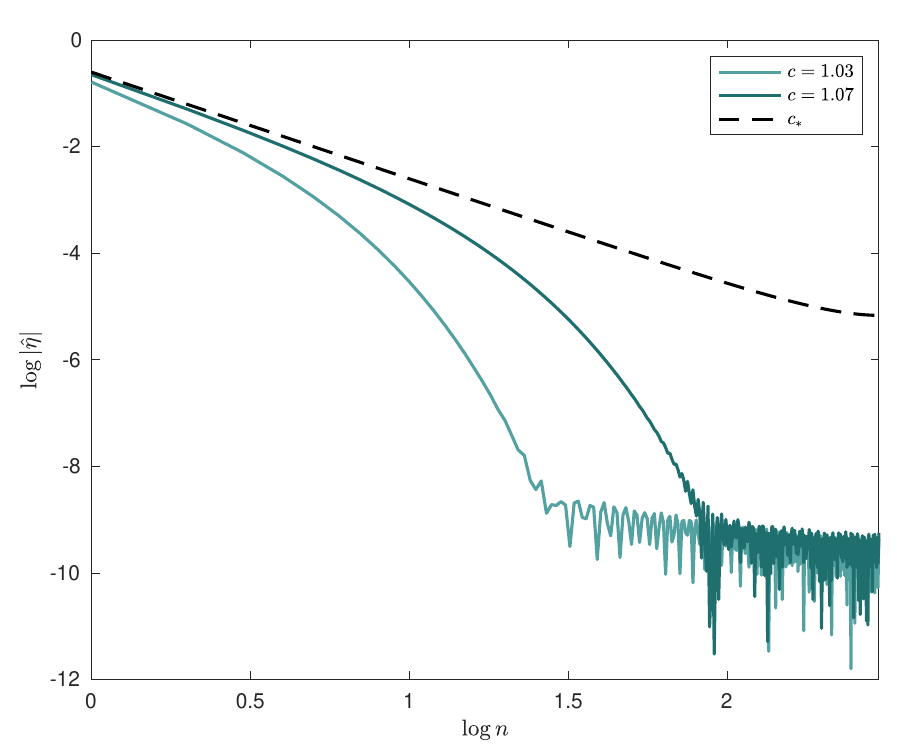} 
	\caption{The solution profiles $\hat{\eta}$ in Fourier space (\ref{eta_n}) in log-log coordinates for $c = 1.03,1.07$ with $N=300$ grid points and $\epsilon = 10^{-14}$ tolerance. The black dashed line represents the peaked profile $\eta_*$ for $c = c_*$.}
	\label{fig-2}
\end{figure}

The solution $\{(x_j,\eta_j)\}_{j=-N}^{N}$ can be represented in Fourier space by using the discrete Fourier transform (DFT) with $2N+1$ modes:
\begin{equation}\label{eta_n}
\hat \eta_n = \frac{h}{2\pi} \sum_{j=-N}^{N-1}\eta_j e^{-in x_j},\quad \quad n \in \{-N,\dots, N\}
\end{equation}
where one of the end point $x_N=\pi$ in the physical space is removed due to  the $2\pi$-periodicity. For the peaked wave profile $\eta_*$ at $c=c_*$, the solution (\ref{quadratic}) can be represented as the Fourier cosine series
\begin{equation}\label{quadratic-Fourier}
\eta_* (x)= -\frac{\pi^2}{48}+\sum_{m=1}^{\infty}\frac{\cos(mx)}{2m^2}.
\end{equation}
By applying DFT on the selected grid, the solution points $\{(n,\hat{\eta}_n)\}_{n=0}^{N}$ are plotted for the smooth profiles in Figure \ref{fig-2}. The black dashed line represents the Fourier series (\ref{quadratic-Fourier}) fot the peaked profile. We note that the fast convergence of the Fourier transform for the smooth waves is replaced by the slow convergence $\mathcal{O}(m^{-2})$ of the Fourier transform for the peaked waves.

Next we study numerically the spectrum of the Hessian operator 
$\mathcal{L}: H^2_{\rm per}(\mathbb{T}) \subset L^2(\mathbb{T}) \to L^2(\mathbb{T})$ which appears in the linearized equation (\ref{linear-local}). The spectrum of $\mathcal L$ can be computed by solving the spectral problem with the $2\pi$-periodic conditions, 
\begin{equation}\label{eigvalue-prob}
\mathcal L \gamma = \lambda \gamma, \quad \gamma \in H^2_{\rm per}(\mathbb{T}).
\end{equation}
We will apply two numerical methods (the finite--difference method and the Fourier collocation method) to solve the spectral problem (\ref{eigvalue-prob}). We write $\mathcal L =\mathcal M + \mathcal W$ with  $\mathcal M = -\partial_x (c^2-2\eta)\partial_x$ and $\mathcal W= (2\eta'' -1)$.

For the finite difference method, using the numerical approximation of the solution profile $\{(x_j,\eta_j)\}_{j=-N}^N$ and the central difference approximation of the second derivative $\mathcal M$, we construct the differentiation matrix for $\mathcal L$ acting on the eigenvector 
$\gamma = (\gamma_{-N},\dots , \gamma_{N-1})\in\mathbb R^{2N}$, 
{\small \begin{equation*}
\mathcal L = \begin{bmatrix}
\mathcal (W^0 +\mathcal  M^0) (\eta_{-N})  & \mathcal M^{+1}(\eta_{-N}) & 0 & \cdots & \mathcal M^{-1}(\eta_N) \\
\mathcal M^{-1}(\eta_{-N+1}) & (\mathcal W^0 +\mathcal  M^0)(\eta_{-N+1}) & \mathcal M^{+1}(\eta_{-N+1})  & \cdots & 0 \\
0 & \mathcal M^{-1}(\eta_{-N+2}) & \mathcal (W^0 +\mathcal  M^0) (\eta_{-N+2})  &  \cdots & 0 \\
\vdots & \vdots & \vdots & \ddots & \vdots \\
\mathcal M^{+1}(\eta_{N-1}) & 0 & 0 & \cdots & (\mathcal W^0 +\mathcal  M^0)(\eta_{N-1}) 
\end{bmatrix}, 
\end{equation*}
}where the boundary point $x_N = \pi$ is removed due to the $2\pi$-periodicity. The diagonal elements $\mathcal{M}^0(\eta_j)$, $\mathcal W^0(\eta_j)$ and the off-diagonal elements $\mathcal M^{\pm 1}(\eta_j)$ for $j\in \{-N,\dots ,N-1\}$ can be written as
\begin{equation*}
\mathcal M^{0}(\eta_j)= \tfrac{2c^2-2\eta_j-\eta_{j+1}-\eta_{j-1} }{h^2}, \quad 
\mathcal M^{\pm 1}(\eta_j)= -\tfrac{c^2-\eta_j-\eta_{j\pm1} }{h^2},
\end{equation*}
and 
\begin{equation*}
\mathcal W^0(\eta_j) =  
	\frac{4\mathcal E + 2\eta_j^2 - 2c^2\eta_j}{\left (c^2-2\eta_j \right )^2}-1
\end{equation*}
for $c\in (1,c_*)$, where the differential equations (\ref{TW-eq}) and (\ref{TW-inv}) have been used for $\eta''(x)$. By numerically solving the eigenvalue problem (\ref{eigvalue-prob}), we obtain the first four eigenfunctions $\gamma$ plotted in Figure \ref{fig-3}. The eigenfunctions display spikes near $x = 0$ in the  limit of $c \to c_*$.

\begin{figure}[htp!]
	\centering
	\includegraphics[width=0.9\textwidth]{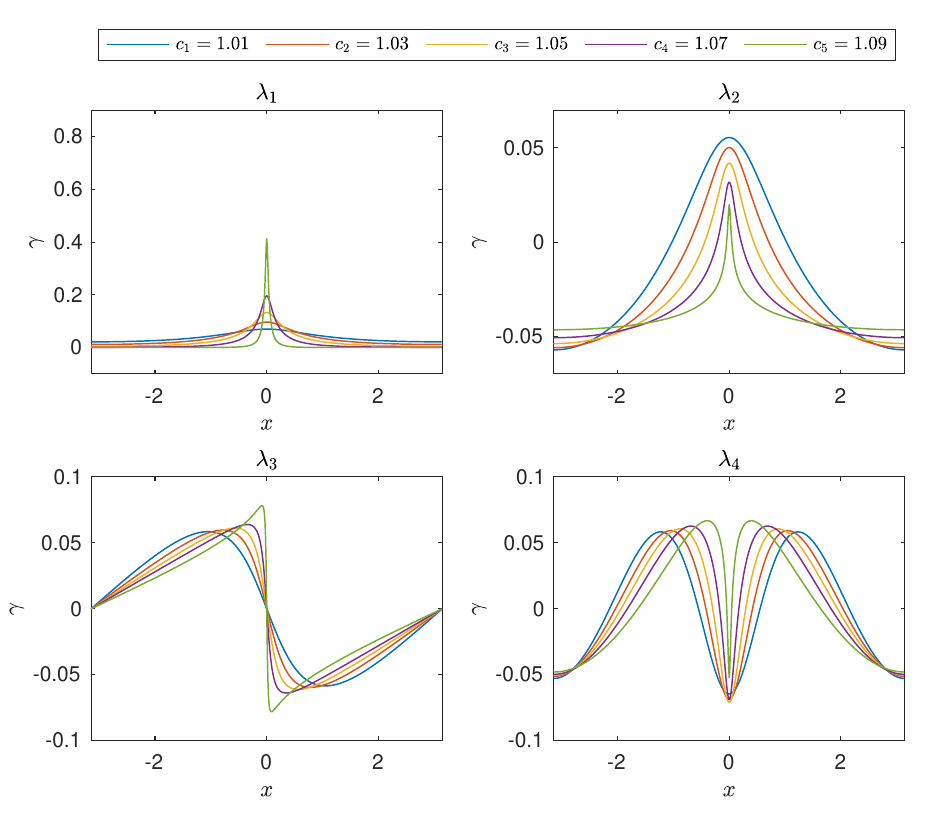} 
	\caption{Eigenfunctions corresponding to the first four eigenvalues for five values of $c$ in $(1,c_*)$. The grid in physical space is chosen to be $N=300$. The solution profiles obtained from equation (\ref{newton}) are used for diagonalization, and all eigenfunctions are plotted on $[-\pi,\pi]$ with positive slope near $-\pi$. }
	\label{fig-3}
\end{figure}

For the Fourier collocation method, we use the discrete Fourier transform (\ref{eta_n}) and represent $\gamma$ in the spectral problem (\ref{eigvalue-prob}) by $\hat \gamma = (\hat \gamma_{-N}, \dots, \hat \gamma_{N})\in \mathbb R^{2N+1}$. Since the $L^2$-isomorphism between physical and Fourier space, the eigenvalue problem in Fourier space $\widehat {\mathcal L} \hat \gamma = \lambda \hat \gamma$ shares the same eigenvalues with the physical space. We write again $\widehat {\mathcal L}= \widehat{\mathcal M} + \widehat{\mathcal W}$ with 
\begin{equation*}
\widehat{\mathcal M}=2(\mathcal D_1\hat \eta) \mathcal D_1-c^2\mathcal D_2+ 2\hat \eta \mathcal D_2, \qquad
\widehat{\mathcal W} =
2\mathcal (D_2\hat \eta) - \pi I 
\end{equation*}
for $c \in (1,c_*)$, where the first and second derivative are represented by  
\begin{equation*}
\mathcal D_1 = i\,\mathrm{diag}(-N,\dots, N),\quad \quad \mathcal D_2 = \mathcal D_1^2
\end{equation*}
and $\hat{\eta}$ is the Toeplitz matrix for convolution with the Fourier modes for $m \in \{-N,\dots, N\}$, 
\begin{equation*}
\hat \eta = \begin{bmatrix}
\hat \eta_0 &  \cdots & \hat \eta_{-N} & \cdots & 0 \\
\vdots &  \ddots & \vdots  &  \ddots &  \vdots  \\
\hat \eta_N  &  \cdots & \hat \eta_0 &  \cdots & \hat \eta_{-N} \\
\vdots &  \ddots & \vdots &  \ddots & \vdots \\
0  & \cdots & \hat \eta_N &  \cdots & \hat \eta_0
\end{bmatrix}.
\end{equation*}
By numerically solving the eigenvalue problem (\ref{eigvalue-prob}) in the Fourier space, we obtain the first four eigenfunctions $\gamma$ plotted in Figure \ref{fig-4}. Convergence of eigenfunctions in Fourier space for large $m$ becomes worse as $c \to c_*$. 

\begin{figure}[htb!]
	\centering
	\includegraphics[width=0.9\textwidth]{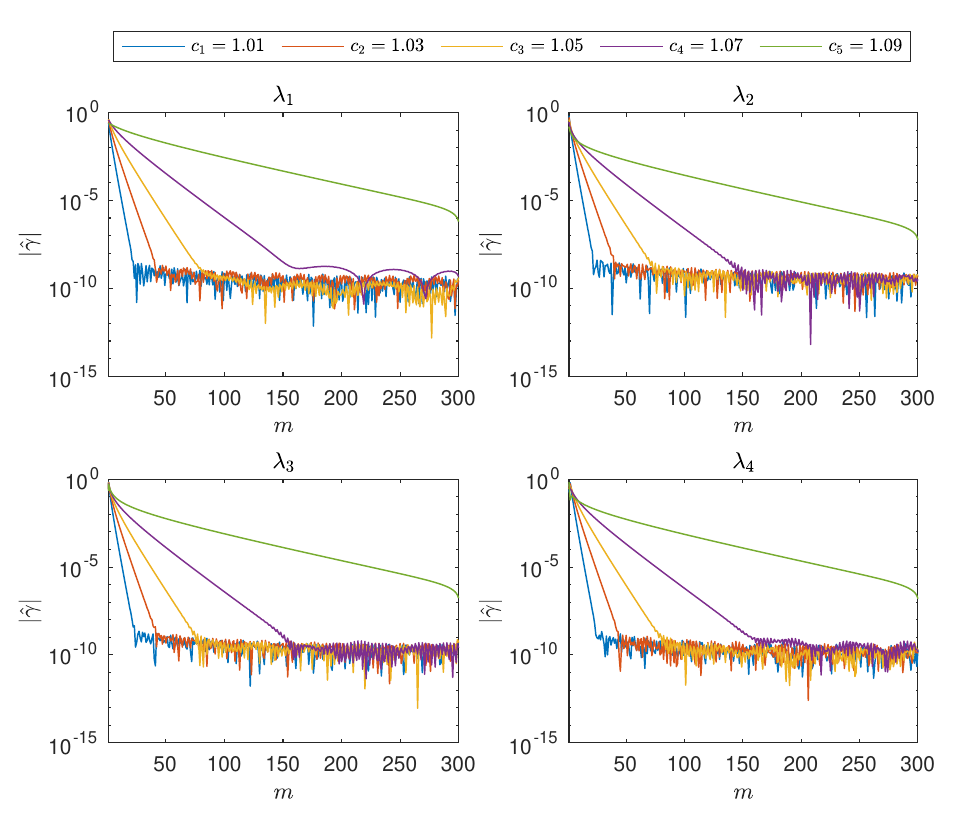} 
	\caption{The absolute value of eigenfunctions corresponding to the first four eigenvalues in Fourier space is plotted versus $m\in \{1,\dots, N\}$ for five values of $c$ in $(1,c_*)$. The grid in physical space is chosen to be $N=300$, and the solution profiles $\hat \eta$ are obtained from equations (\ref{newton}) and (\ref{eta_n}).}
	\label{fig-4}
\end{figure}

Figure \ref{fig-5} shows the first four eigenvalues obtained by the finite-difference method (left) and the Fourier collocation method (right). The lowest eigenvalue diverges to $-\infty$ as $c \to c_*$. After $\lambda_1$, the even-numbered eigenvalues $\lambda_2, \lambda_4, \dots$ correspond to eigenfunctions of even parity in $x$, whereas the odd-numbered eigenvalues $\lambda_3,\lambda_5,\dots$ correspond to eigenfunctions of odd parity.
Convergence of eigenvalues as $c \to c_*$ is low in both physical and Fourier space and, in particular, their values do not converge well to eigenvalues computed for the limiting wave with the peaked profile $\eta_*$ at $c = c_*$. 
The grey shaded region highlights the loss of precision in the numerical approximations of eigenvalues with poor convergence to the eigenvalues of the limiting peaked wave.

\begin{figure}[htp!]
	\centering
	\includegraphics[width=0.9\textwidth]{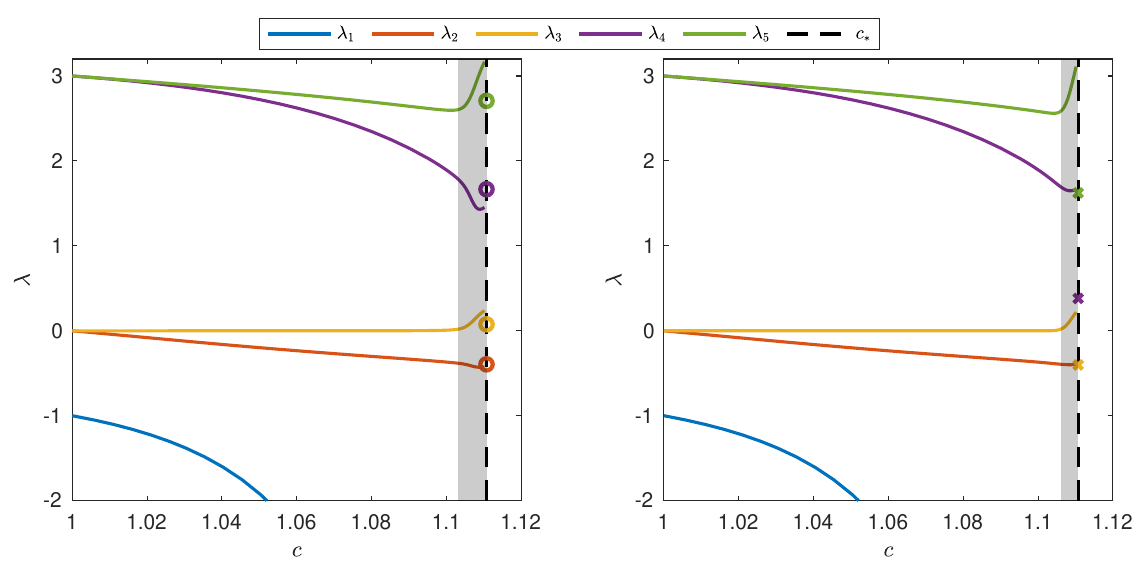} 
	\caption{The dependence of the first four eigenvalues of the spectral problem (\ref{eigvalue-prob}) is plotted versus $c$ for $c \in (1,c_*)$ obtained with the finite-difference method (left) and with the Fourier collocation method (right). Eigenvalues computed for the peaked profile with $c = c_*$ are marked as circles \& crosses. }
	\label{fig-5}
\end{figure}

Eigenvalues for the limiting wave with the peaked profile $\eta_*$ at $c = c_*$ are computed as follows. For the finite-difference method, we use 
\begin{equation*}
c = c_* : \quad \mathcal{W}^0(\eta_j) = -\frac{1}{2} - \pi (\delta_0)_j 
\end{equation*}
where the Dirac delta distribution is approximated by the Gaussian pulse with a small parameter $\alpha = \pi/N$ as 
$$
\delta_0(x) \approx \frac{1}{\sqrt{\pi \alpha^2}}e^{-x^2/\alpha^2}, \quad x \in \mathbb{T}.
$$
For the Fourier collocation method, we use 
$$
c = c_* : \quad \widehat{\mathcal W} = -\frac{1}{2} I - \pi \mathbb I,
$$
where $\mathbb{I}$ is the Toeplitz matrix of unity:
\begin{equation*}
\mathbb I = \begin{bmatrix}
1 &  \cdots & 1& \cdots & 0 \\
\vdots &  \ddots & \vdots  &  \ddots &  \vdots  \\
1 &  \cdots & 1&  \cdots & 1 \\
\vdots &  \ddots & \vdots &  \ddots & \vdots \\
0  & \cdots & 1 &  \cdots & 1
\end{bmatrix}.
\end{equation*}

To illustrate further the low convergence of eigenvalues as $c \to c_*$, we plot the dependence of the third eigenvalue (which is  theoretically zero, see \cite{LP24}) versus $c$ in Figure \ref{fig-6} (left). For computations with different methods and for different $N$, we observe the growth $|\lambda_3|$, which is a numerical way to detect the loss of accuracy of numerical computations. Similarly, Figure \ref{fig-6} (right) shows the difference between eigenvalues computed in the two numerical methods versus $c$. The difference grows as $c \to c_*$ due to low accuracy in each numerical method.

\begin{figure}[htp!]
	\centering
	\includegraphics[width=0.9\textwidth]{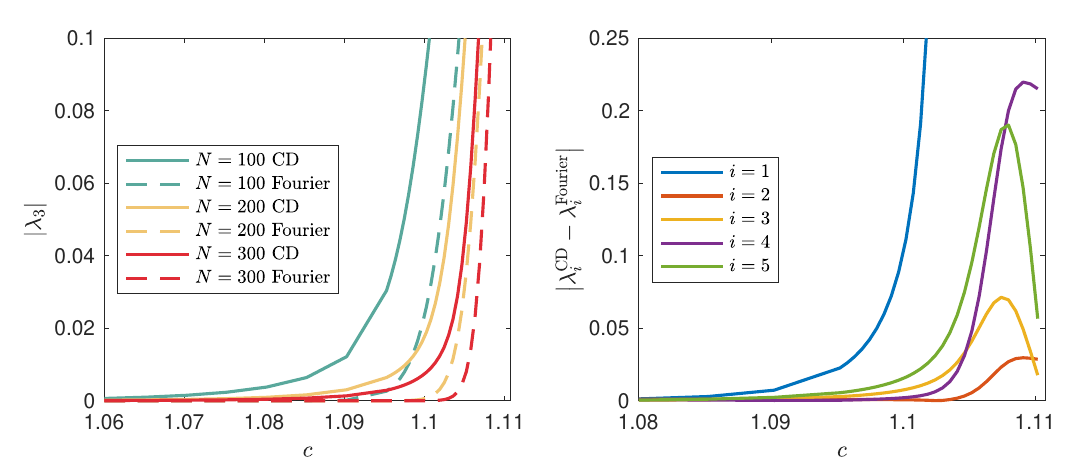} 
	\caption{(a) The third eigenvalue $|\lambda_3|$ plotted versus $c$ to show its departure from $0$ as $c\to c_*^-$ for different grids $N=100,200,300$ by the finite difference (CD) and Fourier collocations (Fourier) methods. (b) Errors between each eigenvalue computed in the two numeical methods versus $c$.}
	\label{fig-6}
\end{figure}

We conclude that the spectral stability problem for the traveling wave with the peaked profile $\eta_* \in C^0_{\rm per}(\mathbb{T}) \cap W^{1,\infty}(\mathbb{T})$ cannot be analyzed by working with the spectral stability problem for the family of traveling waves with the smooth profiles 
$\eta \in C_{\rm per}^{\infty}(\mathbb{T})$ in the limit $c \to c_*$. The lack of convergence is fundamental, both at the levels of functional analysis and 
numerical approximations.

\end{document}